%% file: iryskt14.tex
\documentclass[11pt]{amsart}
\usepackage{amsmath,amsfonts,amssymb,amsthm,amscd,latexsym,euscript,verbatim}
\usepackage[all]{xy}
\usepackage{graphicx}
\usepackage{color}
\makeatletter
\def\section{\@startsection{section}{1}%
  \z@{1.1\linespacing\@plus\linespacing}{.8\linespacing}%
  {\normalfont\Large\scshape\centering}}
\makeatother

\theoremstyle{plain}

\newtheorem*{conj*}{Root Groups Conjecture}
\newtheorem*{thm1.2}{(1.2) Theorem}
\newtheorem*{thm1.3}{(1.3) Theorem}
\newtheorem*{thm1.4}{(1.4) Theorem}
\newtheorem*{prop*}{Proposition}

\newtheorem{prop}{Proposition}[section]
\newtheorem{thm}[prop]{Theorem}

\newtheorem{lemma}[prop]{Lemma}

\theoremstyle{definition}
\newtheorem{Def}[prop]{Definition}
\newtheorem{hypothesis}[prop]{Hypothesis}
\newtheorem*{Def*}{Definition}

\newtheorem*{notation*}{Notation}
\newtheorem{remark}[prop]{Remark}
\newtheorem{remarks}[prop]{Remarks}

\newcommand{\cala}{\mathcal{A}}

\newcommand{\calg}{\mathcal{G}}

\newcommand{\zz}{\mathbb{Z}}

\newcommand{\ga}{\alpha}
\newcommand{\gb}{\beta}
\newcommand{\gc}{\gamma}

\newcommand{\gd}{\delta}

\newcommand{\gre}{\epsilon}
\newcommand{\gl}{\lambda}

\newcommand{\gr}{\rho}

\newcommand{\charc}{{\rm char}}

\newcommand{\sminus}{\smallsetminus}
\newcommand{\lan}{\langle}
\newcommand{\ran}{\rangle}

\numberwithin{equation}{section}

\begin{document}
\title[A non-split sharply $2$-transitive group]{A sharply $2$-transitive group without a non-trivial abelian normal subgroup}
\author[Eliyahu Rips, Yoav Segev, Katrin Tent]{Eliyahu Rips$^1$\qquad Yoav Segev\qquad Katrin Tent}

\address{Eliyahu Rips\\
       Einstein Institute of Mathematics\\
        Hebrew University \\
        Jerusalem 91904\\
        Israel}
\email{eliyahu.rips@mail.huji.ac.il}
\thanks{$^1$This research was partially supported by the Israel Science Foundation}

\address{Yoav Segev \\
         Department of Mathematics \\
         Ben-Gurion University \\
         Beer-Sheva 84105 \\
         Israel}
\email{yoavs@math.bgu.ac.il}

\address{Katrin Tent \\
         Mathematisches Institut \\
         Universit\"at M\"unster \\
	 Einsteinstrasse 62\\
         48149 M\"unster \\
         Germany}
\email{tent@wwu.de}


\begin{abstract} 
We show that any group $G$ is contained in some sharply 2-transitive 
group $\calg$ without a non-trivial
abelian normal subgroup.
This answers a long-standing open question.  The involutions
in the groups $\calg$  that we construct have no fixed points.
\end{abstract}

\date{\today}
\maketitle
 
\section{Introduction}

The \emph{finite} sharply $2$-transitive groups were classified by Zassenhaus
in 1936 \cite{Z} and it is known that any finite
sharply $2$-transitive group contains a non-trivial abelian normal subgroup.
 
In the infinite situation no classification is known (see \cite[Problem 11.52, p.~52]{MK}).
It was a long standing open problem
whether every infinite sharply $2$-transitive group contains a non-trivial abelian normal subgroup.
In \cite{Ti} 
Tits proved that this holds for locally compact connected sharply $2$-transitive groups.
Several other papers showed that under certain special conditions the assertion holds
(\cite{BN, GMS, GlGu, M, T2, Tu, W}). 
The reader may wish to consult Appendix \ref{app A} for more detail,
and for a description of our main results using permutation group theoretic language.

An equivalent formulation to the above problem is 
whether every near-domain is a near-field (see \cite{Hall, K, SSS} and  
Appendix \ref{app A} below).     

We here show that this is not the case.  We construct a sharply $2$-transitive infinite group
without a non-trivial abelian normal subgroup.
In fact, the construction is similar in flavor to the free completion of partial generalized polgyons \cite{T1}.

We are grateful to Joshua Wiscons for pointing out an instructive counterexample
to a first version of this paper, and for greatly simplifying parts of the proof 
in a later version.  We are also grateful to Avinoam Mann for greatly simplifying 
the proof of Proposition \ref{prop A1 fp} and for drawing our attention to a point in the
proof that needed correction. We thank all the referees 
of this paper for carefully reading the manuscript
and making very useful remarks that helped to improve the exposition.

Recall that a proper subgroup $A$ of a group
$G$ is {\it malnormal} in $G$ if\linebreak
$A\cap g^{-1}Ag=1,$ for all $g\in G\sminus A$.

\begin{thm}\label{thm main}
Let $G$ be a group with a malnormal subgroup $A$ and an involution $t\in G\sminus A$ such that
$A$ contains no involutions.
Then for any two elements $u, v\in G$ with $Au\ne Av$ there exist
\begin{itemize}
\item[(a)]
an extension $G\le G_1$;

\item[(b)]
a malnormal subgroup $A_1$ of $G_1$ such that $A_1$ does not contain involutions and satisfies $A_1\cap G=A$;

\item[(c)] an element $f\in G_1$ such that $A_1f=A_1u$ and $A_1tf=A_1v.$
\end{itemize}
\end{thm}

\begin{remark}
It is easy to see (see \S2) that in Theorem \ref{thm main}
we may assume that $u=1, v\notin AtA$  and that either: 
(1)\ $v^{-1}\notin AvA$ or: (2)\ $v$ is an involution.
If case (1) holds we take $G_1=G*\lan f\ran$
to be the free product of $G$ with an infinite cyclic group generated by $f,$
and $A_1=\lan A, f, tfv^{-1}\ran$.  If case (2) holds we take
$G_1=\lan G, f\mid f^{-1}tf=s\ran$ and HNN extension and $A_1=\lan A, f\ran$.
\end{remark}

As a corollary to Theorem \ref{thm main} we get the following.

\begin{thm}\label{thm s2t in char 2}
Let $G$ be a group with a malnormal subgroup $A$ such that
$A$ contains no involutions.
Assume further that $G$ is \texttt{not} sharply $2$-transitive on the set of right cosets $A\backslash G$.
Then $G$ is contained in a group $\calg$ having a malnormal subgroup $\cala$ such that
\begin{enumerate}
\item
$\cala\cap G=A;$

\item
$\calg$ is sharply $2$-transitive on the set of right cosets $X:=\cala\backslash\calg;$

\item
$\cala$ contains no involutions (i.e.~$\calg$ is of permutational characteristic $2$);

\item
$\calg$ does not contain a non-trivial abelian normal subgroup;

\item
if $G$ is infinite then $G$ and $\calg$ have the same cardinality (similarly 
for $X$ and $A\backslash G$).
\end{enumerate}
\end{thm}

As an immediate consequence of Theorem \ref{thm s2t in char 2} we have

\begin{thm}\label{thm eg}
Any group $G$ is contained in a group $\calg$ acting sharply $2$-transitively
on a set $X$ such that each involution in $\calg$ has no fixed point in $X,$ and such that
$\calg$ does not contain a non-trivial abelian normal subgroup.
\end{thm}
\begin{proof}
For $|G|=1,2$ this is obvious.  Otherwise take $A=1$ in Theorem \ref{thm s2t in char 2}.  
\end{proof}
\noindent
In fact there are many other ways to obtain a group $\calg$ having a malnormal subgroup 
$\cala$ and satisfying (2)--(4) of Theorem \ref{thm s2t in char 2},
e.g., take $G=\lan t\ran* A,$ where $t$ is an involution,
and $A$ a non-trivial group without involutions, and apply Theorem \ref{thm s2t in char 2}.
(Here the free product guarantees that $A$ is malnormal in $G$.)

Theorem \ref{thm eg} shows that there exists a sharply $2$-transitive
group $\calg$ of {\it characteristic $2$} (see Definition \ref{def char} in Appendix \ref{app A})
such that $\calg$ does not contain a non-trivial abelian normal subgroup.  Further
as noted in Appendix \ref{app A}, if $G$ is sharply $2$-transitive of characteristic $3,$
then $G$ contains a non-trivial abelian normal subgroup.  The cases where $\charc(G)$ is distinct from
$2$ and $3$ remain open.

Finally we mention that the hypothesis that $A$ does not contain involutions
in Theorem \ref{thm main} is used only in the case where we take $G_1$ to be
an HNN extension of $G$, and then, it is used only
in the proof of the malnormality of $A_1$ in $G_1$.

\section{Some preliminaries regarding Theorem \ref{thm main}}\label{sect explanation}

The following observations and remarks are here in order to explain 
to the reader the way we intend to prove Theorem \ref{thm main}, and
to explain the main division between the two cases  we deal with in \S\ref{sect nonhnn} and \S\ref{sect hnn}.

In fact Lemma \ref{lem hyp}(3) and Lemma \ref{lem fuv} below, together with Remark \ref{rem pf thm 1.1}, show that we may
assume throughout this paper that hypothesis \ref{hyp main} holds; and that hypothesis naturally
leads to the division of the two cases dealt with in \S\ref{sect nonhnn} and \S\ref{sect hnn}.

\begin{lemma}\label{lem hyp}
Let $A$ be a malnormal subgroup of a group $G$ and let $g\in G\sminus A$. Then
\begin{enumerate}
\item
$C_G(a)\le A,$ for all $a\in A,\ a\ne 1;$ 

\item
$\lan g\ran\cap A=1;$

\item
$AgA$ contains an involution iff $g^{-1}\in AgA.$
\end{enumerate}
\end{lemma}
\begin{proof}
(1): Let $a\in A$ with $a\ne 1,$ and let $h\in C_G(a)$. Then $a\in A\cap A^h$. 
So $h\in A,$ since $A$ is malnormal in $G$.
\medskip

\noindent
(2):  Since $g\in C_G(g^k)$ for all integers $k,$  part (2) follows from (1).
\medskip

\noindent
(3):\quad
If $g^{-1}\notin AgA,$ then clearly $AgA$ does not contain an involution.
Conversely, assume that $g^{-1}\in AgA$.  Then $g^{-1}=agb,$ for some $a, b\in A,$  
so $(ag)^2=ab^{-1}\in A$.  Then, by (2), either $(ag)^2=1$ or $ag\in A$.  But $g\notin A,$
so $ag\notin A,$ and we have $(ag)^2=1$.  Hence $AgA$ contains the involution $ag$.
\end{proof}
 
We now make the following observation (and introduce the following notation):
 
\begin{lemma}\label{lem fuv}
Let $G$ be a group with a malnormal subgroup $A$ and an involution $t\in G\sminus A$.
Let $G_1$ be an extension of $G,$ such that
$G_1$ contains a malnormal subgroup $A_1$ with $A_1\cap G=A$. Let $r, s\in G$ be such that $Ar\ne As$. Then
\begin{itemize}
\item[(1)]
there is at most one element $f'\in G_1$ with $A_1f'=A_1r$ and $A_1tf'=A_1s$, which
we denote by $f'=f_{r,s}$ (if it exists).
\end{itemize}
The convention in $(2)$--$(4)$ below is that the left side exists if and only if the right side does
and then they are equal:
\begin{itemize}
\item[(2)]
$f_{r,s}g=f_{rg, sg}$ for any $g\in G.$

\item[(3)]
$tf_{r,s}=f_{s,r}.$

\item[(4)]
$f_{a_1r, a_2s}=f_{r,s}$ for all $a_1, a_2\in A$.
\end{itemize}  
\end{lemma}
\begin{proof}
(1):\quad
Let $f_1, f_2\in G_1$ such that $A_1f_1=A_1f_2=A_1r$ and $A_1tf_1=A_1tf_2=A_1s$.  Then $f_1f_2^{-1}\in A_1$ 
and $tf_1f_2^{-1}t\in A_1$.  Since $t\in G_1\sminus A_1,$ and  since
$A_1$ is malnormal in $G_1,$  we obtain that $f_1f_2^{-1}=1,$ so $f_1=f_2$.
\medskip

\noindent
(2):\quad 
$A_1f_{r,s}g=A_1rg=A_1f_{rg, sg}$ and $A_1tf_{r,s}g=A_1sg=A_1f_{rg,sg}$.
So, by (1), $f_{rg, sg}=f_{r,s}g$.
\medskip

\noindent
(3):\quad
$A_1tf_{r,s}=A_1s,$ and $A_1ttf_{r,s}=A_1f_{r,s}=A_1r$. So, by (1), $tf_{r,s}=f_{s,r}$.
\medskip

\noindent
(4):\quad
$A_1f_{a_1r, a_2s}=A_1a_1r=A_1r,$ and $A_1tf_{a_1r,a_2s}= A_1a_2s=A_1s$. So, by (1), $f_{a_1r, a_2s}=f_{r,s}$. 
\end{proof}

\begin{remark}\label{rem pf thm 1.1}
Let the notation be as in Theorem \ref{thm main}.  Notice that if there is an element 
$f\in G$ such that $Af=Au$ and $Atf=Av,$ we can just take  $G_1=G$
and $A_1=A$ and there is nothing to prove in Theorem \ref{thm main}.  

Hence we may assume throughout this paper that this is not the case. 
In view of (2) and (4)
of Lemma \ref{lem fuv}, $f_{u,v}=f_{1,vu^{-1}}u,$ and $f_{1,a'va}=f_{a^{-1},a'v}a=f_{1,v}a,$ for $a, a'\in A$.
Hence we may assume that $u=1$ (and hence $v\notin A$) and replace $v$ by any element of
the double coset $AvA$.  By Lemma \ref{lem hyp}(3), we may assume that either $v^{-1}\notin AvA,$
or $v$ is an involution.  Further, since $f_{1,t}=1$ and since $t$ is an involution, we may assume
that $v\notin AtA$ and $v^{-1}\notin AtA$.
\end{remark}

Hence it suffices to prove Theorem \ref{thm main} under the following
hypothesis which we assume for the rest of the paper.

\begin{hypothesis}\label{hyp main}
In the setting of Theorem \ref{thm main}, assume $u=1, v, v^{-1}\notin AtA$ and either $v^{-1}\notin AvA$ or $v$ is an involution.
\end{hypothesis}

\section{The case  $v^{-1}\notin AvA$}\label{sect nonhnn}

The purpose of this section is to prove Theorem \ref{thm main} of
the introduction  in the case where $v^{-1}\notin AvA$. We refer
the reader to Hypothesis \ref{hyp main} and to its explanation in \S\ref{sect explanation}.
Thus, throughout this section we assume that $v^{-1}\notin AvA$.  
Also, throughout
this section we use the notation and hypotheses of Theorem \ref{thm main}.

Let $\lan f_1\ran$ be an infinite cyclic group.
We let
\[
G_1=G*\lan f_1\ran,\quad f_2=tf_1v^{-1},\quad A_1=\lan A, f_1, f_2\ran.
\]

In this section we will prove the following theorem.

\begin{thm}\label{thm main nonhnn}
We have
\begin{enumerate}
\item
$A_1=A*\lan f_1\ran*\lan f_2\ran,$ with $f_1, f_2$ of infinite order;

\item
$A_1$ is malnormal in $G_1$.
\end{enumerate}
\end{thm}

Suppose Theorem \ref{thm main nonhnn} is proved.  We now prove Theorem \ref{thm main}
in the case where $v^{-1}\notin AvA$.
\medskip

\begin{proof}[Proof of Theorem \ref{thm main} in the case where $v^{-1}\notin AvA$]\hfill

Let $f:=f_1$.  Then $A_1f=A_1f_1=A_1,$ and 
\[
A_1tf=A_1tf_1=A_1tf_1v^{-1}v=A_1f_2v=A_1v.
\]
By Theorem  \ref{thm main nonhnn}(2), $A_1$ is malnormal in $G_1$.  
By Theorem \ref{thm main nonhnn}(1),  
$A_1\cap G=A,$ and $f_2$ is of infinite order.  
Since $A_1=A*\lan f_1\ran*\lan f_2\ran,$ and $A$ does not contain involutions, $A_1$
does not contain involutions.  
\end{proof}

\begin{prop}\label{prop A1 fp}
$f_2$ is of infinite order in $G_1,$ and $A_1=A* \lan f_1\ran*\lan f_2\ran.$
\end{prop}
\begin{proof}
We first show that $f_2$ is of infinite order.  Indeed let $h:=f_2^n,$ for some
$n\in\zz,$ and write $h$ in terms of $f_1$ and elements of $G$.  If $n>0,$
then $h$ starts with $t$ and ends with $v^{-1},$ while if $n<0,$ then
$h$ starts with $v$ and ends with $t$.  In particular $f_2$ has infinite order.

Next let $F:=\lan f_1, f_2\ran$.  Then any element of $F$ is a product of
alternating powers of $f_1$ and $f_2$.  As we saw in the previous paragraph of the proof,
any non-zero power of $f_2$ starts with $t$ or $v$ and ends with $t$ or $v^{-1}$.
Since $G_1=G*\lan f_1\ran$ there will be no cancellation between powers of $f_1$ and  powers of $f_2$.
It follows that $F$ is a free group.

Now consider an element in $A_1=\lan A, F\ran$.  It is an alternating product
of elements of $A$ and elements of $F$.  
When we express it as an element of $G_1=G*\lan f_1\ran,$ $f_2$ is written as $tf_1v^{-1}$
and $f_2^{-1}$ is written as $vf_1^{-1}t$.  Accordingly, an element $1\ne a\in A$ in this alternating
product is multiplied with $1, v^{-1}$ or $t$ on the left, and with $1, t$ or $v$ on the right.  
The possibilities
are: 
\begin{itemize}
\item
$v^{-1}a,\ ta,\ at,\ av:$ all are distinct from $1$ since $t$ and $v$ are not in $A$.

\item
$tat,\ v^{-1}av:$ all are distinct from $1$ since they are conjugate to $a$.

\item
$tav,\ v^{-1}at:$ all are distinct from $1$ since $v\notin AtA$.\qedhere
\end{itemize}
\end{proof}

\begin{prop}\label{prop malnor nonhnn}
$A_1$ is a malnormal subgroup of $G_1$.
\end{prop}
\begin{proof}
We will show that the existence of elements $a, b\in A_1,$ and $g\in G_1\sminus A_1,$
such that $a\ne 1$ and $g^{-1}ag=b$ leads to a contradiction.

Let 
\[
a=a_1f_{\gd_1}^{\gre_1}a_2f_{\gd_2}^{\gre_2}\cdots a_nf_{\gd_n}^{\gre_n}a_{n+1},\ a\ne 1,\qquad\text{and}
\]
\[
b=b_1f_{\gc_1}^{\mu_1}b_2f_{\gc_2}^{\mu_2}\cdots b_{\ell}f_{\gc_{\ell}}^{\mu_{\ell}}b_{\ell+1},
\]
where $a_i, b_j\in A,\ \gre_i, \mu_j=\pm 1,\ \gd_i, \gc_j\in\{1,2\},$
 and if $\gd_i=\gd_{i-1}$ and $\gre_i=-\gre_{i-1}$ then
$a_i\ne 1$ (i.e.~there are no $f_i$-cancellations in $a$),
and similarly there are no $f_i$-cancellations in $b$.
Write
\[
g=g_1f_1^{\gl_1}g_2f_1^{\gl_2}\cdots g_mf_1^{\gl_m}g_{m+1}\in G_1\sminus A_1,
\]
where $g_i\in G,\ \gl_i=\pm 1,$ and there are no $f_1$-cancellations in $g$.

Assume that $m$ is the least possible.
We have the picture as in Figure \ref{fig9} below.
\begin{figure}[h]
    \centering
\scalebox{0.70}{\input{bild9.pspdftex}}
    \caption{}
    \label{fig9}
\end{figure} 
%

\noindent
{\bf Case 1.}\ $m=n=0$.

In this case $b=g^{-1}a_1g \in A_1\cap G$. By Proposition \ref{prop A1 fp},  $A_1\cap G=A,$ so $b\in A,$
and we get a contradiction to the malnormality of $A$ in $G$.
 
\medskip

\noindent  
\smallskip

The next case to consider is: 
\smallskip

\noindent   
{\bf Case 2.}\ $m=0,$ and $n>0$.

Since $G_1=G*\lan f_1\ran,$ we must have $n=\ell$.
Consider Figure \ref{fig9}.  
By an analysis of the normal form in the free
product $G*\lan f_1\ran$ we see that
the only way we can get the equality $g_1^{-1}ag_1=b$ is 
if both $\gre_1=\mu_1$ and $\gre_n=\mu_n$.  
We distinguish a number of cases as follows.
\begin{itemize}
\item[(i)]
$\gd_1=\gc_1$ or $\gd_n=\gc_n$.

\item[(ii)]
$\gd_1\ne\gc_1$ and $\gd_n\ne \gc_n$.
\begin{itemize}
\item[(a)] 
$n=1$.
\item[(b)]
$n>1$.
\end{itemize}
\end{itemize}
{\bf Case (i).}\ 
By symmetry we may consider only the case where $\gd_1=\gc_1$.  
In this case, regardless of whether $\gre_1=1$ or $-1$ and whether $\gd_1=1$ or $2,$
we get that $g_1=a_1b_1^{-1}\in A,$ a contradiction.  
\medskip

\noindent
{\bf Case (iia).}\
By symmetry we may assume that $\gd_1=1$ and $\gc_1=2$.

Suppose first that $\gre_1=\mu_1=1$.
Then from the left side of Figure \ref{fig9} we get
$a_1^{-1}g_1b_1t=1,$ and from the right side we get $a_2g_1b_2^{-1}v=1$.
This implies that $t\in Ag_1A$ and $v^{-1}\in Ag_1A$.  But then $v^{-1}\in AtA,$
a contradiction.

Suppose next that $\gre_1=\mu_1=-1$.  Then, from the left side of Figure \ref{fig9} we get
$a_1^{-1}g_1b_1v=1,$ and from the right side we get $a_2g_1b_2^{-1}t=1$.
Again this implies that $v^{-1}\in AtA,$
a contradiction.
\medskip

\noindent
{\bf Case (iib).}\ 
By symmetry, we may assume without loss of generality that 
\[
\gd_1=1\text{ and }\gc_1=2.
\]
Suppose first that 
\[
\gre_1=\mu_1=1.
\]
We may further assume that 
\[
a_1^{-1}g_1b_1t=1\quad\text{and}\quad\gre_2=\mu_2.
\]
We now separate the discussion according to the following cases: 
\begin{itemize}
\item
$\gd_2=\gc_2$.  In this case, regardless of
the sign of $\gre_2=\mu_2$ and whether $\gd_2=\gc_2=1$ or $2,$ we get that $a_2^{-1}v^{-1}b_2=1,$
which is false since $v\notin A$.

\item
$\gre_2=\mu_2=1,\ \gd_2=1,\ \gc_2=2$.  We get $a_2^{-1}v^{-1}b_2t=1,$ contradicting $v\notin AtA$.

\item
$\gre_2=\mu_2=-1,\ \gd_2=1,\ \gc_2=2$.  We get $a_2^{-1}v^{-1}b_2v=1$ with $b_2\ne 1$.
But this contradicts the malnormality of $A$ in $G$.

\item
$\gre_2=\mu_2=1,\gd_2=2,\ \gc_2=1$.  We get $ta_2^{-1}v^{-1}b_2=1,$ contrary to $v^{-1}\notin AtA$.

\item
$\gre_2=\mu_2=-1,\ \gd_2=2,\ \gc_2=1$. We get $v^{-1}a_2^{-1}v^{-1}b_2=1$.  This implies that $v^{-1}\in AvA,$
contrary to our hypotheses.
\end{itemize}

Suppose next that
\[
\gre_1=\mu_1=-1.
\]
We may further assume that 
\[
a_1^{-1}g_1b_1v=1\quad\text{and}\quad \gre_2=\mu_2.
\]
Again we separate the discussion according to the following cases: 
\begin{itemize}
\item
$\gd_2=\gc_2$.  In this case, regardless of
the sign of $\gre_2=\mu_2$ and whether $\gd_2=\gc_2=1$ or $2,$ we get that $a_2^{-1}tb_2=1,$
which is false since $t\notin A$.

\item
$\gre_2=\mu_2=1,\ \gd_2=1,\ \gc_2=2$.  We get $a_2^{-1}tb_2t=1,$ and $b_2\ne 1$.  This contradicts 
the malnormality of $A$ in $G$.

\item
$\gre_2=\mu_2=-1,\ \gd_2=1,\ \gc_2=2$.  We get $a_2^{-1}tb_2v=1,$ impossible, as above.

\item
$\gre_2=\mu_2=1,\gd_2=2,\ \gc_2=1$.  We get $ta_2^{-1}tb_2=1$.
This case forces $a_2=b_2=1$ (because $A$ is malnormal in $G$) .  If $n=2$ we get 
$v^{-1}a_3g_1b_3^{-1}=1$.  But this together with $a_1^{-1}g_1b_1v=1$
implies that $v^{-1}\in AvA,$ contrary to our hypotheses.  Thus $n\ge 3$.  But now,
we must have $\gre_3=\mu_3,$ and arguing exactly as in the previous cases, 
for all choices of $\gre_3=\mu_3, \gd_3$ and $\gc_3,$ we get a contradiction as in one of the cases above.

\item
$\gre_2=\mu_2=-1,\ \gd_2=2,\ \gc_2=1$. We get $v^{-1}a_2^{-1}tb_2=1,$ impossible, as above.
\end{itemize}
\smallskip

Next we consider:
\smallskip

\noindent
{\bf Case 3.}\ $n=0=\ell$ and $m > 0$.

\noindent
Notice that in this case there will be no cancellations in Figure \ref{fig9},
since otherwise we must either have $g_1^{-1}a_1g_1=1,$ or $g_{m+1}^{-1}b_1^{-1}g_{m+1}=1,$
which is false.

Hence we may assume that either $n>0$ or $\ell>0$ or both.
By symmetry we may consider the following case:
\smallskip

\noindent
{\bf Case 4.}\ $m > 0$ and $n >0$.

Notice that $f_i$-cancellations have to occur in the product $g^{-1}agb^{-1},$ since
it is equal to $1$.  Now $f_i$-cancellations can occur
only if  one of the following cases occurs:
\begin{itemize}
\item[(i)]
The product $f_1^{-\gl_1}g_1^{-1}a_1f_{\gd_1}^{\gre_1}$ equals $1,\ v^{-1},$ or $t$.
\item[(ii)]
The product $f_{\gd_n}^{\gre_n}a_{n+1}g_1f_1^{\gl_1}$ equals $1,\ t$ or $v$.
\item[(iii)]
The product $f_1^{\gl_m}g_{m+1}b_1f_{\gc_1}^{\mu_1}$ equals $1,\ v^{-1}$ or $t$.
\item[(iv)]
The product $f_{\gc_{\ell}}^{\mu_{\ell}}b_{\ell+1}g_{m+1}^{-1}f_1^{-\gl_m}$
equals $1,\ t$ or $v$.
\end{itemize}
\medskip

By symmetry, we may consider only case (i).
If $f_1^{-\gl_1}g_1^{-1}a_1f_{\gd_1}^{\gre_1}=1,$
then $g_1f_1^{\gl_1}=a_1f_{\gd_1}^{\gre_1}$.  Let 
$h:=g_2f_1^{\gl_2}\cdots f_1^{\gl_m}g_{m+1},$ and
\[
a'=a_2f_{\gd_2}^{\gre_2}\cdots f_{\gd_n}^{\gre_n}a_{n+1}g_1f_1^{\gl_1}=a_2f_{\gd_2}^{\gre_2}\cdots f_{\gd_n}^{\gre_n}a_{n+1}a_1f_{\gd_1}^{\gre_1}\in A_1.
\]
Notice that $a'$ is conjugate to $a,$ so $a'\ne 1$.
Also $h=f_1^{-\gl_1}g_1^{-1}g,$ and $h\notin A_1,$ since $f_1^{-\gl_1}g_1^{-1}\in A_1,$ while $g\notin A_1$.
We get (see Figure \ref{fig9}) $g^{-1}ag=h^{-1}a'h\in A_1,$ contradicting the minimality of $m$.
\smallskip

If $f_1^{-\gl_1}g_1^{-1}a_1f_{\gd_1}^{\gre_1}=v^{-1},$ then $g_1f_1^{\gl_1}=a_1f_{\gd_1}^{\gre_1}v$.
Let $h:=vg_2f_1^{\gl_2}\cdots f_1^{\gl_m}g_{m+1},$ and 
\[
a'=a_2f_{\gd_2}^{\gre_2}\cdots f_{\gd_n}^{\gre_n}a_{n+1}g_1f_1^{\gl_1}v^{-1}=a_2f_{\gd_2}^{\gre_2}\cdots f_{\gd_n}^{\gre_n}a_{n+1}a_1f_{\gd_1}^{\gre_1}\in A_1.
\]
As above, $1\ne a'\in A_1,$ and if $h\in A_1,$ then 
$g=g_1f_1^{\gl_1}v^{-1}h=a_1f_{\gd_1}^{\gre_1}h\in A_1,$
which is false.
We again get $g^{-1}ag=h^{-1}a'h\in A_1,$ which contradicts the minimality of $m$.

Finally if $f_1^{-\gl_1}g_1^{-1}a_1f_{\gd_1}^{\gre_1}=t,$ then $g_1f_1^{\gl_1}=a_1f_{\gd_1}^{\gre_1}t$.
Let $h:=tg_2f_1^{\gl_2}\cdots f_1^{\gl_m}g_{m+1},$ and 
\[
a'=a_2f_{\gd_2}^{\gre_2}\cdots f_{\gd_n}^{\gre_n}a_{n+1}g_1f_1^{\gl_1}t=a_2f_{\gd_2}^{\gre_2}\cdots f_{\gd_n}^{\gre_n}a_{n+1}a_1f_{\gd_1}^{\gre_1}\in A_1.
\]
As above we get $1\ne a'\in A_1,$ and $h\notin A_1,$
and  again we get the same contradiction.

Note that if $\ell=0,$ then no cancellation of the type (iii) or (iv) above can occur.
\end{proof}
\begin{proof}[Proof of Theorem \ref{thm main nonhnn}]\hfill
\medskip

\noindent
By Proposition \ref{prop A1 fp}, part (1) holds, and by Proposition \ref{prop malnor nonhnn} part (2) holds.
\end{proof}
\section{The case $v$ is an involution and $v\notin AtA$}\label{sect hnn}
The purpose of this section is to prove Theorem \ref{thm main} of
the introduction  in the case where $v$ is an involution. We refer
the reader to Hypothesis \ref{hyp main} and to its explanation in \S\ref{sect explanation}.
Thus, throughout this section we assume that $v$ is an involution and that $v\notin AtA$.  Further, throughout
this section we use the notation and hypotheses of Theorem \ref{thm main}.

Let $\lan f\ran$ be an infinite cyclic group.
We define an HNN extension
\[
G_1=\lan G, f\mid f^{-1}tf=v\ran,\qquad A_1=\lan A, f\ran.
\]

In this section we will prove the following theorem.
 
\begin{thm}\label{thm main hnn}
We have
\begin{enumerate}
\item
$A_1=A*\lan f\ran$;

\item
$A_1$ is malnormal in $G_1.$
\end{enumerate}
\end{thm}
Suppose Theorem \ref{thm main hnn} is proved.  We now use it to prove Theorem \ref{thm main}
in the case where $v$ is an involution. 
\smallskip

\noindent
\begin{proof}[Proof of Theorem \ref{thm main} in the case where $v$ is an involution]\hfill
\medskip

\noindent
We have $A_1f=A_1$ and $A_1tf=A_1fv=A_1v$.  By Theorem \ref{thm main hnn}(2), $A_1$
is malnormal in $G_1$.  By Theorem \ref{thm main hnn}(1), $A_1\cap G=A$. 
Also $A_1$ does not
contain involutions since $A_1=A*\lan f\ran,$ and $A$ does not contain involutions.
\end{proof}

\begin{remark}\label{rem hnn}
Any element of $G_1$ has the form 
\[
g=g_1f^{\gd_1}g_2\cdots g_mf^{\gd_m}g_{m+1},
\]
where $g_i\in G,\ i=1,\dots m+1,\ \gd_i=\pm 1,\ i=1,\dots m$.
According to Britton's lemma we say that there are {\it no $f$-cancellations in $g$} if the equality
$\gd_i=-\gd_{i-1}$ implies that if $\gd_i=1,$ then $g_i\ne 1, t,$ while if $\gd_i=-1,$ then
$g_i\ne 1,v$.
 
Further let $g$ be as above, let $h\in G_1,$ and write:
\[
h=h_1f^{\eta_1}h_2\cdots h_kf^{\eta_k}h_{k+1},
\]
where $h_j\in G,\ j=1,\dots k+1,\ \eta_j=\pm 1,\ j=1,\dots k,$  
and there are no $f$-cancellations in $g$ and $h$.

Then $g=h$    
if and only if $m=k,\ \gd_i=\eta_i,\ i=1,\dots, m,$
and there are elements $w_0,z_1,w_1,z_2,w_2,\dots,z_m,w_m,z_{m+1}$ such that
for every oriented loop in Figure \ref{fig1} the product of edges is $1,$ that is:

\begin{itemize}
\item[(a)]
$h_i=w_{i-1}g_iz_i,\ i=1,\dots, m+1;$

\item[(b)]  $w_0=1,\ z_{m+1}=1;$

\item[(c)]  if $\gd_i=1,$ then either $z_i=1,\ w_i=1,$ or $z_i=t,\ w_i=v;$

\item[(d)]  if $\gd_i=-1,$ then either $z_i=1,\ w_i=1,$ or $z_i=v,\ w_i=t.$
\end{itemize}
\clearpage

\begin{figure}[h]
    \centering
\scalebox{0.70}{\input{bild8.pspdftex}}
    \caption{}
    \label{fig1}
\end{figure}  
\end{remark}
\begin{lemma}\label{A1 is a free product}
$A_1=A*\lan f\ran$.
\end{lemma}
\begin{proof}
Suppose that
\[
g_1f^{\gd_1}g_2\cdots g_mf^{\gd_m}g_{m+1}=h_1f^{\gd_1}h_2\cdots h_mf^{\gd_m}h_{m+1},
\]
and $h_i, g_i\in A,\ i=1,\dots, m+1$. By Remark \ref{rem hnn}, $h_1=g_1z_1,$
hence, by (a)--(d) of Remark \ref{rem hnn}, since $t, v\notin A,$ we have $z_1=1,$ 
so $h_1=g_1,$ and then, by Remark \ref{rem hnn} (c) and (d), $w_1=1$.

Assume $w_i=1$.  Then $h_{i+1}=w_ig_{i+1}z_{i+1}=g_{i+1}z_{i+1}$.  
Since $t, v\notin A,$ this implies $z_{i+1}=1,$ and
then $w_{i+1}=1$.  So $g_i=h_i$ for $i=1,\dots, m+1$.  Hence $A_1=A*\lan f\ran$.
\end{proof}

\begin{prop}\label{prop malnormal hnn}
$A_1$ is malnormal in $G_1$. 
\end{prop}
\begin{proof}
We will show that the existence of elements $a, b\in A_1,\ g\in G_1\sminus A_1$
such that $a\ne 1$ and $g^{-1}ag=b$ leads to a contradiction. Let
\[
a=a_1f^{\ga_1}a_2\cdots a_mf^{\ga_m}a_{m+1},\qquad b=b_1f^{\gb_1}b_2\cdots b_nf^{\gb_n}b_{n+1},
\]
where $a_i, b_i\in A,\ \ga_i, \gb_i=\pm 1,$ and if $\ga_i=-\ga_{i-1},$ then $a_i\ne 1,$
and if $\gb_i=-\gb_{i-1},$ then $b_i\ne 1$.  Recall that by Lemma \ref{A1 is a free product}, $A_1=A*\lan f\ran,$
and therefore in the above expressions for $a$ and $b$ there are no $f$-cancellations.  We also have
\[
g=g_1f^{\gd_1}g_2\cdots g_kf^{\gd_k}g_{k+1},
\]
where $g_i\in G,\ \gd_i=\pm 1,$ and $\gd_i=-\gd_{i-1}$ implies that if $\gd_i=1,$ then $g_i\ne 1, t,$
and if $\gd_i=-1,$ then $g_i\ne 1, v$.

We assume that $k$ is the least possible.
\medskip

\noindent
{\bf Case 1.}\ $k=0$.

Then $g=g_1,$ so we have 
\[
g_1^{-1}a_1f^{\ga_1}a_2\cdots a_mf^{\ga_m}a_{m+1}g_1=b_1f^{\gb_1}b_2\cdots b_nf^{\gb_n}b_{n+1}.
\]  
We conclude that $n=m,\ \ga_i=\gb_i,$ for $i=1,2,\dots, m$.  If $m=n=0,$ then $a=a_1\ne 1,\ b=b_1,$ so 
$g_1^{-1}a_1g_1=b_1$ which is impossible because $A$ is malnormal in $G$.

Let $m=n>0$.  We obtain Figure \ref{fig2} below, where
\begin{gather}\label{eq prop 4.4}
\text{if }\ga_i=1,\text{ then either }p_i=q_i=1,\text{ or }p_i=t,\ q_i=v,\\\notag
\text{and if }\ga_i=-1,\text{ then either }p_i=q_i=1,\text{ or }p_i=v,\ q_i=t.
\end{gather}

\begin{figure}[h]
    \centering
\scalebox{0.70}{\input{bild1.pspdftex}}
    \caption{}
    \label{fig2}
\end{figure}  

\noindent
We have $p_1=a_1^{-1}g_1b_1\notin A$ since $g_1\notin A$.  Now 
assume $p_i\notin A$.  Then $q_i\notin A$ by \eqref{eq prop 4.4} 
and by Britton's Lemma $p_{i+1}=a_{i+1}^{-1}q_ib_{i+1}$ is not in $A$ either.
In Particular $p_i, q_i\ne 1$ for all $i\le m$.
%
%

If $m=n\ge 2,$ consider Figure \ref{fig3}:
\begin{figure}[h]
    \centering
\scalebox{0.70}{\input{bild2.pspdftex}}
    \caption{}
    \label{fig3}
\end{figure} 

\noindent
We now use equation \eqref{eq prop 4.4}.
If $\ga_1=1,\ \ga_2=1,$ then $q_1=v,\ p_2=t,$ so $v=a_2tb_2^{-1}\in AtA,$ a contradiction.

If $\ga_1=1,\ \ga_2=-1,$ then $a_2\ne 1,\ q_1=v,\ p_2=v$.  Then $va_2v=b_2,$
contradicting the malnormality of $A$ in $G$.

If $\ga_1=-1,\ \ga_2=1,$ then $a_2\ne 1,\ q_1=t, p_2=t,$ and $ta_2t=b_2,$ again contradicting the malnormality of $A$ in $G$.

If $\ga_1=-1,\ \ga_2=-1,$ then $q_1=t,\ p_2=v,$ and $v=a_2^{-1}tb_2\in AtA,$
a contradiction.

So we are left with the possibility $m=n=1$.  In Figure \ref{fig2} above, after cutting and pasting
we obtain the following figure \ref{fig4}:
\begin{figure}[h]
    \centering
\scalebox{0.70}{\input{bild3.pspdftex}}
    \caption{}
    \label{fig4}
\end{figure} 

\noindent
If $\ga_1=1,$ then $p_1=t,\ q_1=v,$ and if $\ga_1=-1,$ then $p_1=v,\ q_1=t$.  In
both cases $v\in AtA,$ contrary to  the choice of $v$.
\medskip

\noindent
{\bf Case 2.} $k>0$.

Consider Figure \ref{fig5} below.
\begin{figure}[h]
    \centering
\scalebox{0.70}{\input{bild4.pspdftex}}
    \caption{}
    \label{fig5}
\end{figure} 
 
Notice that $f$-cancellations have to occur in the product $g^{-1}agb^{-1},$ since
it is equal to $1$. Therefore,   at least one of the following cases must happen:
\begin{enumerate}
\item
$m=0,\ a=a_1,$ and $f^{-\gd_1}$ cancels with $f^{\gd_1}$ in the product $f^{-\gd_1}g_1^{-1}a_1g_1f^{\gd_1};$

\item
$n=0,\ b=b_1$ and $f^{\gd_k}$ cancels with $f^{-\gd_k}$ in the product $f^{\gd_k}g_{k+1}b_1g_{k+1}^{-1}f^{-\gd_k};$

\item
$m>0,$ and $f^{-\gd_1}$ cancels with $f^{\ga_1}$ in the product $f^{-\gd_1}g_1^{-1}a_1f^{\ga_1};$

\item
$m>0,$ and $f^{\ga_m}$ cancels with $f^{\gd_1}$ in the product $f^{\ga_m}a_{m+1}g_1f^{\gd_1};$

\item
$n>0,$ and $f^{\gd_k}$ cancels with $f^{\gb_1}$ in the product $f^{\gd_k}g_{k+1}b_1f^{\gb_1};$

\item
$n>0,$ and $f^{\gb_n}$ cancels with $f^{-\gd_k}$ in the product $f^{\gb_n}b_{n+1}g_{k+1}^{-1}f^{-\gd_k}.$ 
\end{enumerate}  
In case (1), $a=a_1\ne 1,$ so $g_1^{-1}a_1g_1=t\text{ or }v$. Hence $a_1$
is conjugate to an involution, which is impossible, as $A$ does not contain involutions.

Similarly, in case (2) $b=b_1\ne 1,$ so $g_{k+1}b_1g_{k+1}^{-1}=t\text{ or }v,$ again a contradiction.

In case (3) we have Figure \ref{fig6} below,
\begin{figure}[h]
    \centering
\scalebox{0.70}{\input{bild5a.pspdftex}}
    \caption{}
    \label{fig6}
\end{figure} 

\noindent
where $p,q\in\{1, t, v\}$ by Britton's Lemma.  We define 
\[
a'=a_2\cdots a_mf^{\ga_m}a_{m+1}a_1f^{\ga_1}\quad\text{ and }\quad h=qg_2\cdots g_kf^{\gd_k}g_{k+1}.
\]
We have $h^{-1}a'h=b,\ a'$ is conjugate to $a$. So $a\ne 1$ implies $a'\ne 1$.  Also the $f$-length
of $h$ is $k-1$.  
Notice that
$h=f^{-\ga_1}a_1^{-1}g,$ and $h\notin A_1$ since $f^{-\ga_1}a_1^{-1}\in A_1,$ and $g\notin A_1$.
We obtained a contradiction to the minimality of $k$.

The remaining cases are handled in entirely the same way. 
\end{proof}
\smallskip

\begin{proof}[Proof of Theorem \ref{thm main hnn}]\hfill
\medskip

\noindent
By Lemma \ref{A1 is a free product}, part (1) holds, and by Proposition \ref{prop malnormal hnn},
part (2) holds.
\end{proof}

\section{The proof of Theorem \ref{thm s2t in char 2}}

In this section we show how Theorem \ref{thm s2t in char 2} of the
introduction follows from Theorem \ref{thm main}.

Let $G$ be a group with a malnormal subgroup $A$ such that $A$ contains no involutions.   
Assume that $G$ is {\it not}\, $2$-transitive on the set of right cosets $A\backslash G$.
If there exists an involution $t\in G\sminus A,$ set $G_0:=G,\ A_0:=A$.
Otherwise, let $G_0:=G*\lan t\ran,$ where $t$ is an involution, and let $A_0=A$.
Then, by  \cite[Corollary 4.1.5]{MaKS}, $G$ is malnormal in $G_0,$ and
then since $A$ is malnormal in $G,$ it is malnormal in $G_0$.

We now construct a sequence of groups $G_i$ and of subgroups
$A_i\le G_i,\ i=0,1,2\dots,$ having the following properties for all $i\ge 0$:
\begin{enumerate}
\item
$G_i\le G_{i+1},$ and $A_i\le A_{i+1};$

\item
$A_i$ is malnormal in $G_i$ and $t\in G_i\sminus A_i;$

\item
$A_i$ does not contain involutions;

\item
$A_{i+1}\cap G_i=A_i$;

\item
for each $v\in G_i\sminus A_i$ there
exists an element $f_v\in A_{i+1}$ such that  $A_{i+1}tf_v=A_{i+1}v.$
\end{enumerate}

In order to construct $G_{i+1}, A_{i+1}$ from $G_i, A_i$
we enumerate the set\linebreak $G_i\sminus A_i=\{v_\alpha:\alpha< \rho\}$
for some ordinal $\rho$.  For each ordinal $\ga<\gr$ we construct the pair 
$G_i^{\ga},\ A_i^{\ga}$ and the element $f_{v_{\ga}}\in A_i^{\ga}$ having
the following properties:
\begin{itemize}
\item[(i)]  $G_i^{\gb}\le G_i^{\ga},$ for all ordinals $\gb<\ga;$

\item[(ii)] $A_i^{\ga}$ is malnormal in $G_i^{\ga}$ and $t\in G_i^{\ga}\sminus A_i^{\ga};$

\item[(iii)]  $A_i^{\ga}$ contains no involutions;

\item[(iv)]  $A_i^{\ga}\cap G_i^{\gb}=A_i^{\gb}$ for all $\gb<\ga;$

\item[(v)]   $f_{v_{\ga}}\in A_i^{\ga}$ and $A_i^{\ga}tf_{v_{\ga}}=A_i^{\ga}v_{\ga}$.
\end{itemize}
We let $G_i^0=G_i$ and $A_i^0=A_i$.  If $\ga=\gb+1,$
we construct $(G_i^{\alpha},\ A_i^{\alpha}, f_{v_{\ga}})$ from  $(G_i^{\gb},\ A_i^{\gb})$
as follows:
If there is some $f\in A_i^{\gb}$ with
$A_i^\gb tf = A_i^\gb v_\alpha$ we let $G_i^{\alpha}=G_i^{\gb},\ $
$A_i^{\alpha}= A_i^{\gb}$ and $f_{v_{\ga}}=f$.
Otherwise apply Theorem \ref{thm main} to  $G_i^\gb,\ A_i^\gb$ with $u=1$ and $v=v_\alpha$
to obtain the groups  $G_i^{\alpha},\ A_i^{\alpha}$ and the element $f_{v_{\ga}}\in A_i^{\ga}$.
Of course, by construction, $A_i^{\ga}$ contains no involutions
and $A_i^{\ga}\cap G_i^{\gb}=A_i^{\gb}$. So (i)--(v) hold.

For a limit ordinal $\alpha$ we put $G_i^{(\ga,1)}=\bigcup_{\beta<\alpha}
G_i^\beta,\ A_i^{(\ga,1)}=\bigcup_{\beta<\alpha} A_i^\beta$.
We now show that  when $\ga$ is a limit ordinal $A_i^{(\ga,1)}$ is malnormal in $G_i^{(\ga,1)}$.
Notice that for each ordinal $\gb<\ga$ and each $g\in G_i^{\gb}\sminus A_i^{\gb},$ 
we have $g\in G_i^{(\ga,1)}\sminus A_i^{(\ga,1)}$.
Indeed else take the minimal $\gc<\ga$ such that $g\in A_i^{\gc}$.  
Then, by definition, $\gc$ is not a limit ordinal,
and $g\in G_i^{\gc-1}\sminus A_i^{\gc-1}$.
So $g\in A_i^{\gc}\cap G_i^{\gc-1}=A_i^{\gc-1},$ 
a contradiction.  This means that $A_i^{(\ga,1)}\cap G_i^{\gb}=A_i^{\gb},$
for all ordinals $\gb<\ga$.

Suppose  now that $g^{-1}ag=b$ with $g\in G_i^{(\ga,1)}\sminus A_i^{(\ga,1)}$ and
$a, b\in A_i^{(\ga,1)}$. Then, by the previous paragraph, there exists $\gb<\ga$ 
so that $a,b\in A_i^{\gb}$
and $g\in G_i^{\gb}\sminus A_i^{\gb}$ and then we get a contradiction
to the malnormality of $A_i^{\gb}$ in $G_i^{\gb}$.  Clearly
$A_i^{(\ga,1)}$ contains no involutions.  Next if there exists $f\in A_i^{(\ga,1)}$
such that $A_i^{(\ga,1)}tf=A_i^{(\ga,1)}u_{\ga}$ then we let 
$G_i^{\ga}=G_i^{(\ga,1)},\ A_i^{\ga}=A_i^{(\ga,1)}$ and $f_{v_{\ga}}=f$.
Else we construct $G_i^{\ga},\ A_i^{\ga}$ and $f_{v_{\ga}}$ from 
$G_i^{(\ga,1)},\ A_i^{(\ga,1)}$ using Theorem \ref{thm main} with
$u=1$ and $v=v_{\ga}$ (just as in the construction above in the case
of a non-limit ordinal).
Again we see that (i)--(v) hold.

Finally put 
\[
G_{i+1}=\bigcup_{\alpha<\rho} G_i^\alpha,\quad A_{i+1}=\bigcup_{\alpha<\rho} A_i^\alpha,
\]
\begin{center}

\end{center}
and set 
\[
\calg=\bigcup_{i<\omega}G_i,\quad \cala=\bigcup_{i<\omega}A_i\quad\text{and}\quad X= \cala\backslash \calg.
\]
As in the construction of $G_i^{(\ga,1)},\ A_i^{(\ga,1)}$
in the case where $\ga$ is a limit ordinal, we see that $\cala$ is malnormal in $\calg$
and that $\cala\cap G_i=A_i,$ for each $i<\omega$.
To see that the action of $G$ on $X$ is $2$-transitive
just note that any $v\in \calg\sminus \cala$ is contained in some $G_i$
so that there is some $f_v\in A_{i+1}\subseteq\cala$ with $A_{i+1} tf_v = A_{i+1}v$.
Since $A_{i+1}\le \cala$ we see that $\cala tf_v=\cala v$ as required.
Since $\cala$ is malnormal in $\calg$ the action of $\calg$ on $X$
is sharply $2$-transitive.  By construction, $\cala$ contains no involutions.

Finally, as is well known, if $\calg$ contains a non-trivial abelian
normal subgroup, then necessarily all involutions in $\calg$
commute with each other (see, e.g., \cite[Remark 4.4]{GMS}).
But, by our construction, this is not the case in $\calg$.
Indeed, if $G_1=G_0*\lan f_1\ran$ is a free product, then $t$ does not
commute with $f_1^{-1}tf_1$.  
Suppose that $G_1=\lan G, f\mid f^{-1}tf=v\ran$ is
an HNN extension. Let $s\in G$ be an involution distinct from $t$
(notice that $t$ is not in the center of $G$ since $A$ is malnormal in $G,$
so such $s$ exists).
Then $sf^{-1}sf$ and $f^{-1}sfs$ are in canonical form, so they are distinct,
and the involutions $s$ and $f^{-1}sf$ do not commute\footnotemark.
\footnotetext{Note that we could start with a group $G_0$  
which already contains an involution that does not commute with $t$.  Then
it would immediately follow that $\calg$ does not split. We thank Uri Bader
for pointing this out.}
This completes the proof of Theorem \ref{thm s2t in char 2}.
\appendix

\section{Some background and a permutation group theoretic point of view}\label{app A} 

Recall that a permutation group $G$ on a set $X$ is {\it regular} if it
is transitive and no non-trivial element of $G$ fixes a point.  $G$ is a {\it Frobenius group} on $X,$
if $G$ is transitive on $X,$ no non-trivial element in $G$ fixes more than one point,
and some non-trivial elements of $G$ fix a point.  $G$ is  
{\it sharply $2$-transitive} if $G$ is transitive on $X,$ and for any two ordered pairs $(x_1, x_2),\ (x_1', x_2')\in X\times X$
of distinct points in $X,$ there exists  a unique element $g\in G$ such that $x_ig=x_i',\ i=1, 2$.

\begin{remarks}
Let $G$ be a group and let $A$ be a subgroup of $G$.  Let $X:=A\backslash G$
be the set of right cosets of $A$ in $G$.  Then the following are equivalent
\begin{enumerate}
\item
$A$ is malnormal in $G$.

\item
Either
\begin{enumerate}
\item
$A=1,$ and $G$ is regular on $X,$ or

\item
$G$ is a Frobenius group on $X$ (so $A\ne 1$).
\end{enumerate}
\end{enumerate}

\end{remarks}

If a sharply $2$-transitive group $G$ on $X$ contains a non-trivial normal abelian subgroup
$B,$ then $B$ is necessarily regular on $X$ and $G=HB$ with $H\cap B=1,$  where $H$ is the stabilizer
in $G$ of some point in $X$.  In this case we say that $G$ {\it splits,}
otherwise we say that $G$ is {\it non-split}.

The primary example of sharply $2$-transitive groups are the {\it $1$-dimensional affine groups}.
Given a field $F,$ the $1$-dimensional affine group over $F$ is the group 
$G:=\{x\mapsto ax+b\mid a, b\in F,\, a\ne 0\}$ of functions on $X=F$.
So $G$ is Frobenius on $X$.  

If $G$ is a $1$-dimensional affine group over $F$, then  $G$ splits.  Indeed
if we let $B=\{x\mapsto x+b\mid b\in F\}$
and $H=\{ax\mid a\in F,\,\, a\ne 0\},$ the stabilizer of $0$ in $G,$ then $B$
is an abelian normal subgroup of $G$ and $G=B H$.  In fact in \S 6 of \cite{K}
it is shown that sharply $2$-transitive groups can
be completely characterized by means of ``one-dimensional affine'' transformations
$x\mapsto ax+b$ on an algebraic structure called a {\it near-domain}
defined in \cite[Definition, p.~21]{K}.  Further, the notion of a {\it near-field}
is defined below the Definition in p.~21 of \cite{K}. And in \cite[Thm.~7.1, p.~25]{K}
it is shown that the assertion that every sharply $2$-transitive group splits is equivalent to the assertion
that {\it every near-domain is a near field} 
(see also \cite[subsection 20.7, p.~382]{Hall}, \cite[chapter 3]{SSS}).

However, for an infinite sharply $2$-transitive group $G$ it was a long standing problem 
whether or not $G$ splits.
It is known that a sharply $2$-transitive group splits in the following cases:
\begin{itemize}
\item
$G$ is locally compact  connected \cite{Ti};

\item
$G$ is locally finite \cite{W};

\item
$G$ is  definable in an o-minimal structure \cite{T2};

\item
$G$ is linear (with certain additional restrictions) \cite{GlGu};

\item
$G$ is locally linear (with some additional restrictions) \cite{GMS};
\item
Further splitting results can be found in \cite{BN} and \cite{SSS}.
\end{itemize}
 
To state some additional splitting results we need to introduce some more definitions.
So let $G$ be an infinite sharply $2$-transitive group on a set $X$.  
Then $G$ contains ``many'' involutions.  Let $I\subset G$ be the set of involutions in $G$.
Then $I$ is a conjugacy class in $G$.  If $i\in I$ has no fixed points in $X$ we say that 
$G$ is of characteristic $2$ and we write $\charc(G)=2$.  Otherwise each $i\in I$ fixes
a unique point.  In this case the set of all products of distinct involutions: $I^2\sminus\{1\}$
form a conjugacy class in $G,$ and a nontrivial power of an element in $I^2\sminus\{1\}$ 
belongs to $I^2\sminus\{1\}$.  It follows that the elements in $I^2\sminus\{1\}$ either
have an odd prime order $p,$ or are of infinite order.  In the former case
we say that the characteristic of $G$ is $p$ and in the latter case we say that the characteristic of $G$ is $0$.
Hence we have the following definition.
\begin{Def}\label{def char}
Let $G$ be a sharply $2$-transitive group on a set $X,$
and let $I$ be the set of involutions in $G$.  Let $I^2=\{ts\mid t, s\in I\}$.
We define the {\it characteristic} of $G,$ denoted $\charc(G)$ as follows:
\begin{itemize}
\item[(char 2)]  $\charc(G)=2$ if $i\in I$ has no fixed point in $X;$ 

\item[(char 0)]  $\charc(G)=0$ if each $g\in I^2\sminus\{1\}$ is of infinite order;

\item[(char p)]  $\charc(G)=p,$ where $p$ is an odd prime, if the order of each 
$g\in I^2\sminus\{1\}$ is $p.$
\end{itemize}
\end{Def}

\begin{itemize}
\item
In \cite[Thm.~9.5, p.~42]{K} and in \cite{Tu} it was shown that if $\charc(G)=3,$ then $G$ splits.

\item
In \cite{M} it was shown that if the exponent of the point stabilizer is $3$ or $6,$ then $G$ splits. 
\end{itemize}

Using the above terminology, we can now rephrase Theorem  \ref{thm s2t in char 2} as follows.
\begin{thm}
Every Frobenius or regular permutation group which is \texttt{not}
sharply $2$-transitive, and  whose involutions do not have any fixed point,
has a non-split sharply $2$-transitive extension of characteristic $2$.
\end{thm}
\noindent
Here, by an ``extension'' we mean an extension of both the given
set and the given permutation group.
\medskip

\subsection*{Acknowledgment.}  
We would like to thank  Martina Pfeifer
for meticulously and efficiently preparing
the figures of this paper.

\end{document}

%% file: bild9.pspdftex
\begin{picture}(0,0)%
\includegraphics{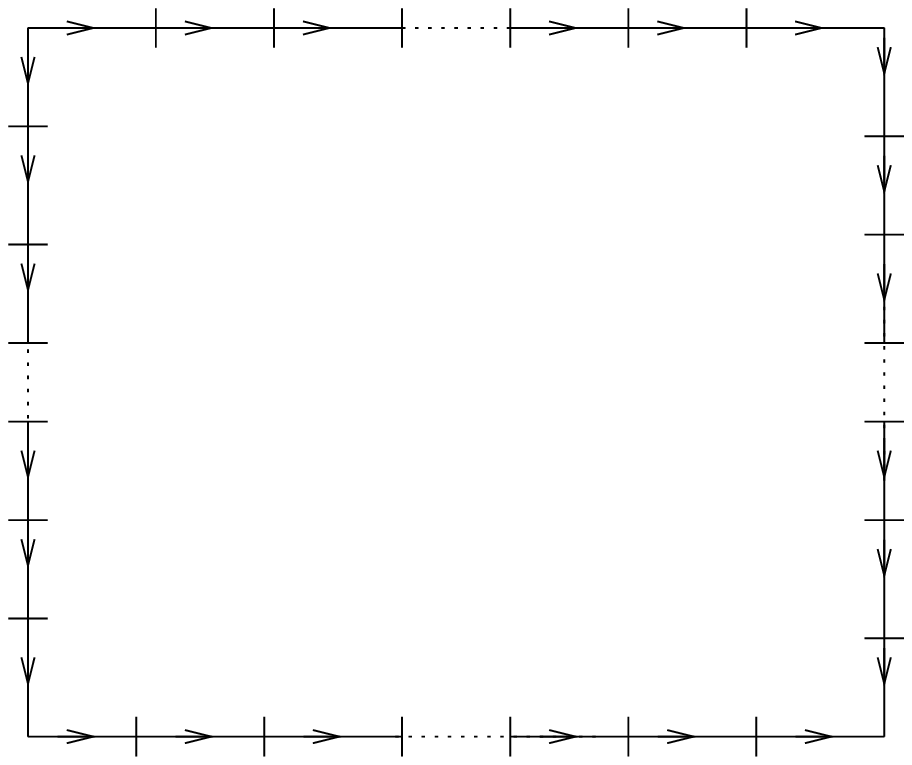}%
\end{picture}%
\setlength{\unitlength}{4144sp}%
\begingroup\makeatletter\ifx\SetFigFont\undefined%
\gdef\SetFigFont#1#2#3#4#5{%
  \reset@font\fontsize{#1}{#2pt}%
  \fontfamily{#3}\fontseries{#4}\fontshape{#5}%
  \selectfont}%
\fi\endgroup%
\begin{picture}(4440,3799)(481,-4634)
\put(1126,-1006){\makebox(0,0)[lb]{\smash{{\SetFigFont{12}{14.4}{\rmdefault}{\mddefault}{\updefault}{\color[rgb]{0,0,0}$a_1$}%
}}}}
\put(2206,-1006){\makebox(0,0)[lb]{\smash{{\SetFigFont{12}{14.4}{\rmdefault}{\mddefault}{\updefault}{\color[rgb]{0,0,0}$a_2$}%
}}}}
\put(3286,-1006){\makebox(0,0)[lb]{\smash{{\SetFigFont{12}{14.4}{\rmdefault}{\mddefault}{\updefault}{\color[rgb]{0,0,0}$a_n$}%
}}}}
\put(3826,-1006){\makebox(0,0)[lb]{\smash{{\SetFigFont{12}{14.4}{\rmdefault}{\mddefault}{\updefault}{\color[rgb]{0,0,0}$f^{\varepsilon_n}_{\delta_n}$}%
}}}}
\put(4411,-1006){\makebox(0,0)[lb]{\smash{{\SetFigFont{12}{14.4}{\rmdefault}{\mddefault}{\updefault}{\color[rgb]{0,0,0}$a_{n+1}$}%
}}}}
\put(4906,-1411){\makebox(0,0)[lb]{\smash{{\SetFigFont{12}{14.4}{\rmdefault}{\mddefault}{\updefault}{\color[rgb]{0,0,0}$g_1$}%
}}}}
\put(4906,-1861){\makebox(0,0)[lb]{\smash{{\SetFigFont{12}{14.4}{\rmdefault}{\mddefault}{\updefault}{\color[rgb]{0,0,0}$f^{\lambda_1}_1$}%
}}}}
\put(4906,-2356){\makebox(0,0)[lb]{\smash{{\SetFigFont{12}{14.4}{\rmdefault}{\mddefault}{\updefault}{\color[rgb]{0,0,0}$g_2$}%
}}}}
\put(4906,-3166){\makebox(0,0)[lb]{\smash{{\SetFigFont{12}{14.4}{\rmdefault}{\mddefault}{\updefault}{\color[rgb]{0,0,0}$g_m$}%
}}}}
\put(4906,-3661){\makebox(0,0)[lb]{\smash{{\SetFigFont{12}{14.4}{\rmdefault}{\mddefault}{\updefault}{\color[rgb]{0,0,0}$f^{\lambda_m}_1$}%
}}}}
\put(4906,-4156){\makebox(0,0)[lb]{\smash{{\SetFigFont{12}{14.4}{\rmdefault}{\mddefault}{\updefault}{\color[rgb]{0,0,0}$g_{m+1}$}%
}}}}
\put(1081,-4561){\makebox(0,0)[lb]{\smash{{\SetFigFont{12}{14.4}{\rmdefault}{\mddefault}{\updefault}{\color[rgb]{0,0,0}$b_1$}%
}}}}
\put(1576,-4561){\makebox(0,0)[lb]{\smash{{\SetFigFont{12}{14.4}{\rmdefault}{\mddefault}{\updefault}{\color[rgb]{0,0,0}$f^{\mu_1}_{\gamma_1}$}%
}}}}
\put(2161,-4561){\makebox(0,0)[lb]{\smash{{\SetFigFont{12}{14.4}{\rmdefault}{\mddefault}{\updefault}{\color[rgb]{0,0,0}$b_2$}%
}}}}
\put(3286,-4561){\makebox(0,0)[lb]{\smash{{\SetFigFont{12}{14.4}{\rmdefault}{\mddefault}{\updefault}{\color[rgb]{0,0,0}$b_l$}%
}}}}
\put(4411,-4561){\makebox(0,0)[lb]{\smash{{\SetFigFont{12}{14.4}{\rmdefault}{\mddefault}{\updefault}{\color[rgb]{0,0,0}$b_{l+1}$}%
}}}}
\put(3826,-4561){\makebox(0,0)[lb]{\smash{{\SetFigFont{12}{14.4}{\rmdefault}{\mddefault}{\updefault}{\color[rgb]{0,0,0}$f^{\mu_l}_{\gamma_l}$}%
}}}}
\put(586,-1366){\makebox(0,0)[lb]{\smash{{\SetFigFont{12}{14.4}{\rmdefault}{\mddefault}{\updefault}{\color[rgb]{0,0,0}$g_1$}%
}}}}
\put(586,-2311){\makebox(0,0)[lb]{\smash{{\SetFigFont{12}{14.4}{\rmdefault}{\mddefault}{\updefault}{\color[rgb]{0,0,0}$g_2$}%
}}}}
\put(586,-3166){\makebox(0,0)[lb]{\smash{{\SetFigFont{12}{14.4}{\rmdefault}{\mddefault}{\updefault}{\color[rgb]{0,0,0}$g_m$}%
}}}}
\put(541,-3616){\makebox(0,0)[lb]{\smash{{\SetFigFont{12}{14.4}{\rmdefault}{\mddefault}{\updefault}{\color[rgb]{0,0,0}$f^{\lambda_m}_1$}%
}}}}
\put(1621,-1006){\makebox(0,0)[lb]{\smash{{\SetFigFont{12}{14.4}{\rmdefault}{\mddefault}{\updefault}{\color[rgb]{0,0,0}$f^{\varepsilon_1}_{\delta_1}$}%
}}}}
\put(496,-4111){\makebox(0,0)[lb]{\smash{{\SetFigFont{12}{14.4}{\rmdefault}{\mddefault}{\updefault}{\color[rgb]{0,0,0}$g_{m+1}$}%
}}}}
\put(541,-1816){\makebox(0,0)[lb]{\smash{{\SetFigFont{12}{14.4}{\rmdefault}{\mddefault}{\updefault}{\color[rgb]{0,0,0}$f^{\lambda_1}_1$}%
}}}}
\end{picture}%

%% file: bild8.pspdftex
\begin{picture}(0,0)%
\includegraphics{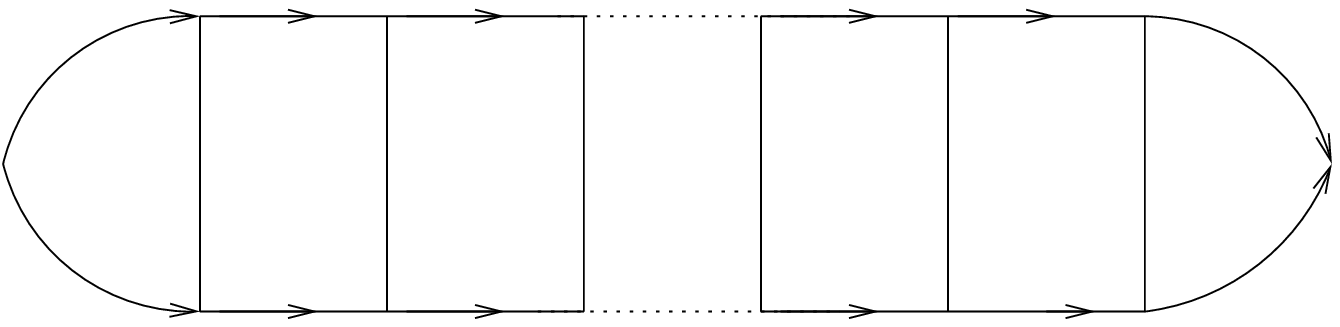}%
\end{picture}%
\setlength{\unitlength}{4144sp}%
\begingroup\makeatletter\ifx\SetFigFont\undefined%
\gdef\SetFigFont#1#2#3#4#5{%
  \reset@font\fontsize{#1}{#2pt}%
  \fontfamily{#3}\fontseries{#4}\fontshape{#5}%
  \selectfont}%
\fi\endgroup%
\begin{picture}(6098,1905)(1426,-3415)
\put(6076,-1681){\makebox(0,0)[lb]{\smash{{\SetFigFont{12}{14.4}{\rmdefault}{\mddefault}{\updefault}{\color[rgb]{0,0,0}$f^{\delta_m}$}%
}}}}
\put(4996,-2446){\makebox(0,0)[lb]{\smash{{\SetFigFont{12}{14.4}{\rmdefault}{\mddefault}{\updefault}{\color[rgb]{0,0,0}$w_{m-1}$}%
}}}}
\put(3331,-2491){\makebox(0,0)[lb]{\smash{{\SetFigFont{12}{14.4}{\rmdefault}{\mddefault}{\updefault}{\color[rgb]{0,0,0}$w_1$}%
}}}}
\put(2431,-2491){\makebox(0,0)[lb]{\smash{{\SetFigFont{12}{14.4}{\rmdefault}{\mddefault}{\updefault}{\color[rgb]{0,0,0}$z_1$}%
}}}}
\put(5896,-2446){\makebox(0,0)[lb]{\smash{{\SetFigFont{12}{14.4}{\rmdefault}{\mddefault}{\updefault}{\color[rgb]{0,0,0}$z_m$}%
}}}}
\put(6751,-2446){\makebox(0,0)[lb]{\smash{{\SetFigFont{12}{14.4}{\rmdefault}{\mddefault}{\updefault}{\color[rgb]{0,0,0}$w_m$}%
}}}}
\put(4141,-2491){\makebox(0,0)[lb]{\smash{{\SetFigFont{12}{14.4}{\rmdefault}{\mddefault}{\updefault}{\color[rgb]{0,0,0}$z_2$}%
}}}}
\put(2656,-3346){\makebox(0,0)[lb]{\smash{{\SetFigFont{12}{14.4}{\rmdefault}{\mddefault}{\updefault}{\color[rgb]{0,0,0}$f^{\delta_1}$}%
}}}}
\put(2701,-1681){\makebox(0,0)[lb]{\smash{{\SetFigFont{12}{14.4}{\rmdefault}{\mddefault}{\updefault}{\color[rgb]{0,0,0}$f^{\delta_1}$}%
}}}}
\put(1441,-2041){\makebox(0,0)[lb]{\smash{{\SetFigFont{12}{14.4}{\rmdefault}{\mddefault}{\updefault}{\color[rgb]{0,0,0}$g_1$}%
}}}}
\put(1441,-3031){\makebox(0,0)[lb]{\smash{{\SetFigFont{12}{14.4}{\rmdefault}{\mddefault}{\updefault}{\color[rgb]{0,0,0}$h_1$}%
}}}}
\put(3556,-1681){\makebox(0,0)[lb]{\smash{{\SetFigFont{12}{14.4}{\rmdefault}{\mddefault}{\updefault}{\color[rgb]{0,0,0}$g_2$}%
}}}}
\put(3511,-3346){\makebox(0,0)[lb]{\smash{{\SetFigFont{12}{14.4}{\rmdefault}{\mddefault}{\updefault}{\color[rgb]{0,0,0}$h_2$}%
}}}}
\put(5221,-1681){\makebox(0,0)[lb]{\smash{{\SetFigFont{12}{14.4}{\rmdefault}{\mddefault}{\updefault}{\color[rgb]{0,0,0}$g_{m}$}%
}}}}
\put(5221,-3346){\makebox(0,0)[lb]{\smash{{\SetFigFont{12}{14.4}{\rmdefault}{\mddefault}{\updefault}{\color[rgb]{0,0,0}$h_{m}$}%
}}}}
\put(7291,-1951){\makebox(0,0)[lb]{\smash{{\SetFigFont{12}{14.4}{\rmdefault}{\mddefault}{\updefault}{\color[rgb]{0,0,0}$g_{m+1}$}%
}}}}
\put(7246,-3031){\makebox(0,0)[lb]{\smash{{\SetFigFont{12}{14.4}{\rmdefault}{\mddefault}{\updefault}{\color[rgb]{0,0,0}$h_{m+1}$}%
}}}}
\put(6121,-3346){\makebox(0,0)[lb]{\smash{{\SetFigFont{12}{14.4}{\rmdefault}{\mddefault}{\updefault}{\color[rgb]{0,0,0}$f^{\delta_m}$}%
}}}}
\end{picture}%

%% file: bild1.pspdftex
\begin{picture}(0,0)%
\includegraphics{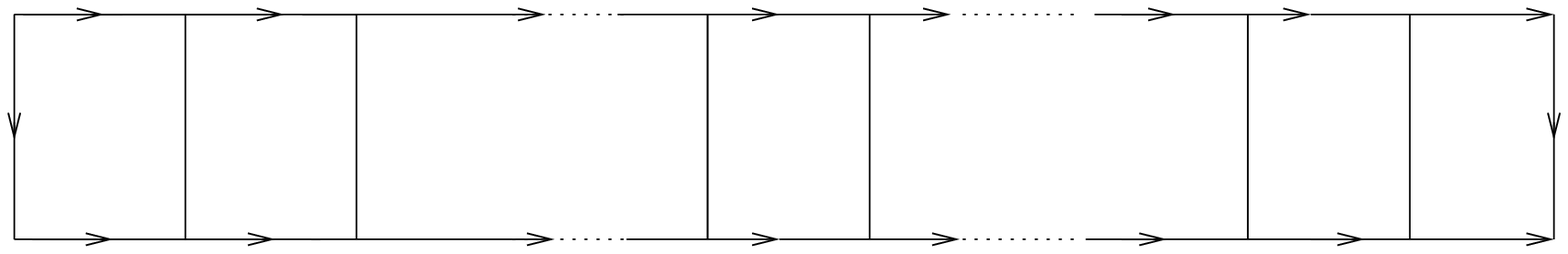}%
\end{picture}%
\setlength{\unitlength}{4144sp}%
\begingroup\makeatletter\ifx\SetFigFont\undefined%
\gdef\SetFigFont#1#2#3#4#5{%
  \reset@font\fontsize{#1}{#2pt}%
  \fontfamily{#3}\fontseries{#4}\fontshape{#5}%
  \selectfont}%
\fi\endgroup%
\begin{picture}(7995,1684)(1066,-3689)
\put(2431,-2176){\makebox(0,0)[lb]{\smash{{\SetFigFont{12}{14.4}{\rmdefault}{\mddefault}{\updefault}{\color[rgb]{0,0,0}$f^{\alpha_1}$}%
}}}}
\put(1441,-2176){\makebox(0,0)[lb]{\smash{{\SetFigFont{12}{14.4}{\rmdefault}{\mddefault}{\updefault}{\color[rgb]{0,0,0}$a_1$}%
}}}}
\put(3016,-2896){\makebox(0,0)[lb]{\smash{{\SetFigFont{12}{14.4}{\rmdefault}{\mddefault}{\updefault}{\color[rgb]{0,0,0}$q_1$}%
}}}}
\put(1486,-3616){\makebox(0,0)[lb]{\smash{{\SetFigFont{12}{14.4}{\rmdefault}{\mddefault}{\updefault}{\color[rgb]{0,0,0}$b_1$}%
}}}}
\put(2431,-3616){\makebox(0,0)[lb]{\smash{{\SetFigFont{12}{14.4}{\rmdefault}{\mddefault}{\updefault}{\color[rgb]{0,0,0}$f^{\alpha_1}$}%
}}}}
\put(5086,-2176){\makebox(0,0)[lb]{\smash{{\SetFigFont{12}{14.4}{\rmdefault}{\mddefault}{\updefault}{\color[rgb]{0,0,0}$f^{\alpha_i}$}%
}}}}
\put(5131,-3616){\makebox(0,0)[lb]{\smash{{\SetFigFont{12}{14.4}{\rmdefault}{\mddefault}{\updefault}{\color[rgb]{0,0,0}$f^{\alpha_i}$}%
}}}}
\put(6976,-2176){\makebox(0,0)[lb]{\smash{{\SetFigFont{12}{14.4}{\rmdefault}{\mddefault}{\updefault}{\color[rgb]{0,0,0}$a_m$}%
}}}}
\put(7021,-3616){\makebox(0,0)[lb]{\smash{{\SetFigFont{12}{14.4}{\rmdefault}{\mddefault}{\updefault}{\color[rgb]{0,0,0}$b_m$}%
}}}}
\put(3331,-3616){\makebox(0,0)[lb]{\smash{{\SetFigFont{12}{14.4}{\rmdefault}{\mddefault}{\updefault}{\color[rgb]{0,0,0}$b_2$}%
}}}}
\put(3286,-2176){\makebox(0,0)[lb]{\smash{{\SetFigFont{12}{14.4}{\rmdefault}{\mddefault}{\updefault}{\color[rgb]{0,0,0}$a_2$}%
}}}}
\put(5626,-2896){\makebox(0,0)[lb]{\smash{{\SetFigFont{12}{14.4}{\rmdefault}{\mddefault}{\updefault}{\color[rgb]{0,0,0}$q_i$}%
}}}}
\put(7156,-2896){\makebox(0,0)[lb]{\smash{{\SetFigFont{12}{14.4}{\rmdefault}{\mddefault}{\updefault}{\color[rgb]{0,0,0}$p_m$}%
}}}}
\put(8281,-2896){\makebox(0,0)[lb]{\smash{{\SetFigFont{12}{14.4}{\rmdefault}{\mddefault}{\updefault}{\color[rgb]{0,0,0}$q_m$}%
}}}}
\put(7696,-3616){\makebox(0,0)[lb]{\smash{{\SetFigFont{12}{14.4}{\rmdefault}{\mddefault}{\updefault}{\color[rgb]{0,0,0}$f^{\alpha_m}$}%
}}}}
\put(1891,-2896){\makebox(0,0)[lb]{\smash{{\SetFigFont{12}{14.4}{\rmdefault}{\mddefault}{\updefault}{\color[rgb]{0,0,0}$p_1$}%
}}}}
\put(1081,-2896){\makebox(0,0)[lb]{\smash{{\SetFigFont{12}{14.4}{\rmdefault}{\mddefault}{\updefault}{\color[rgb]{0,0,0}$g_1$}%
}}}}
\put(9046,-2896){\makebox(0,0)[lb]{\smash{{\SetFigFont{12}{14.4}{\rmdefault}{\mddefault}{\updefault}{\color[rgb]{0,0,0}$g_1$}%
}}}}
\put(7696,-2176){\makebox(0,0)[lb]{\smash{{\SetFigFont{12}{14.4}{\rmdefault}{\mddefault}{\updefault}{\color[rgb]{0,0,0}$f^{\alpha_m}$}%
}}}}
\put(8461,-2176){\makebox(0,0)[lb]{\smash{{\SetFigFont{12}{14.4}{\rmdefault}{\mddefault}{\updefault}{\color[rgb]{0,0,0}$a_{m+1}$}%
}}}}
\put(8461,-3616){\makebox(0,0)[lb]{\smash{{\SetFigFont{12}{14.4}{\rmdefault}{\mddefault}{\updefault}{\color[rgb]{0,0,0}$b_{m+1}$}%
}}}}
\put(4501,-2896){\makebox(0,0)[lb]{\smash{{\SetFigFont{12}{14.4}{\rmdefault}{\mddefault}{\updefault}{\color[rgb]{0,0,0}$p_i$}%
}}}}
\end{picture}%

%% file: bild2.pspdftex
\begin{picture}(0,0)%
\includegraphics{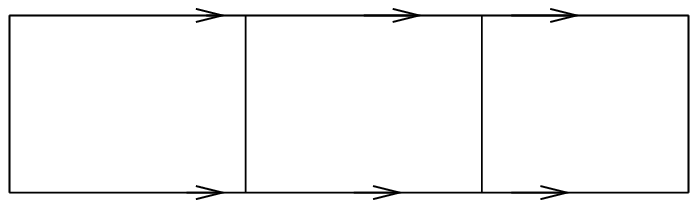}%
\end{picture}%
\setlength{\unitlength}{4144sp}%
\begingroup\makeatletter\ifx\SetFigFont\undefined%
\gdef\SetFigFont#1#2#3#4#5{%
  \reset@font\fontsize{#1}{#2pt}%
  \fontfamily{#3}\fontseries{#4}\fontshape{#5}%
  \selectfont}%
\fi\endgroup%
\begin{picture}(3405,1279)(3676,-3509)
\put(4321,-3436){\makebox(0,0)[lb]{\smash{{\SetFigFont{12}{14.4}{\rmdefault}{\mddefault}{\updefault}{\color[rgb]{0,0,0}$f^{\alpha_1}$}%
}}}}
\put(4321,-2401){\makebox(0,0)[lb]{\smash{{\SetFigFont{12}{14.4}{\rmdefault}{\mddefault}{\updefault}{\color[rgb]{0,0,0}$f^{\alpha_1}$}%
}}}}
\put(5491,-2401){\makebox(0,0)[lb]{\smash{{\SetFigFont{12}{14.4}{\rmdefault}{\mddefault}{\updefault}{\color[rgb]{0,0,0}$a_2$}%
}}}}
\put(3691,-2896){\makebox(0,0)[lb]{\smash{{\SetFigFont{12}{14.4}{\rmdefault}{\mddefault}{\updefault}{\color[rgb]{0,0,0}$p_1$}%
}}}}
\put(7066,-2896){\makebox(0,0)[lb]{\smash{{\SetFigFont{12}{14.4}{\rmdefault}{\mddefault}{\updefault}{\color[rgb]{0,0,0}$q_2$}%
}}}}
\put(6526,-3436){\makebox(0,0)[lb]{\smash{{\SetFigFont{12}{14.4}{\rmdefault}{\mddefault}{\updefault}{\color[rgb]{0,0,0}$f^{\alpha_2}$}%
}}}}
\put(5896,-2896){\makebox(0,0)[lb]{\smash{{\SetFigFont{12}{14.4}{\rmdefault}{\mddefault}{\updefault}{\color[rgb]{0,0,0}$p_2$}%
}}}}
\put(5041,-2896){\makebox(0,0)[lb]{\smash{{\SetFigFont{12}{14.4}{\rmdefault}{\mddefault}{\updefault}{\color[rgb]{0,0,0}$q_1$}%
}}}}
\put(6526,-2401){\makebox(0,0)[lb]{\smash{{\SetFigFont{12}{14.4}{\rmdefault}{\mddefault}{\updefault}{\color[rgb]{0,0,0}$f^{\alpha_2}$}%
}}}}
\put(5491,-3436){\makebox(0,0)[lb]{\smash{{\SetFigFont{12}{14.4}{\rmdefault}{\mddefault}{\updefault}{\color[rgb]{0,0,0}$b_2$}%
}}}}
\end{picture}%

%% file: bild3.pspdftex
\begin{picture}(0,0)%
\includegraphics{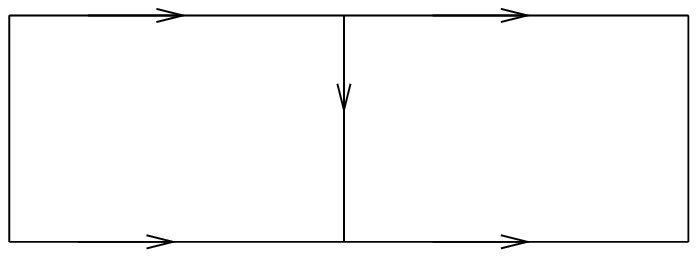}%
\end{picture}%
\setlength{\unitlength}{4144sp}%
\begingroup\makeatletter\ifx\SetFigFont\undefined%
\gdef\SetFigFont#1#2#3#4#5{%
  \reset@font\fontsize{#1}{#2pt}%
  \fontfamily{#3}\fontseries{#4}\fontshape{#5}%
  \selectfont}%
\fi\endgroup%
\begin{picture}(3405,1578)(3496,-3910)
\put(4231,-2491){\makebox(0,0)[lb]{\smash{{\SetFigFont{12}{14.4}{\rmdefault}{\mddefault}{\updefault}{\color[rgb]{0,0,0}$a_2$}%
}}}}
\put(5806,-2491){\makebox(0,0)[lb]{\smash{{\SetFigFont{12}{14.4}{\rmdefault}{\mddefault}{\updefault}{\color[rgb]{0,0,0}$a_1$}%
}}}}
\put(6886,-3166){\makebox(0,0)[lb]{\smash{{\SetFigFont{12}{14.4}{\rmdefault}{\mddefault}{\updefault}{\color[rgb]{0,0,0}$p_1$}%
}}}}
\put(5311,-3166){\makebox(0,0)[lb]{\smash{{\SetFigFont{12}{14.4}{\rmdefault}{\mddefault}{\updefault}{\color[rgb]{0,0,0}$g_1$}%
}}}}
\put(3511,-3166){\makebox(0,0)[lb]{\smash{{\SetFigFont{12}{14.4}{\rmdefault}{\mddefault}{\updefault}{\color[rgb]{0,0,0}$q_1$}%
}}}}
\put(4276,-3841){\makebox(0,0)[lb]{\smash{{\SetFigFont{12}{14.4}{\rmdefault}{\mddefault}{\updefault}{\color[rgb]{0,0,0}$b_2$}%
}}}}
\put(5806,-3841){\makebox(0,0)[lb]{\smash{{\SetFigFont{12}{14.4}{\rmdefault}{\mddefault}{\updefault}{\color[rgb]{0,0,0}$b_1$}%
}}}}
\end{picture}%

%% file: bild4.pspdftex
\begin{picture}(0,0)%
\includegraphics{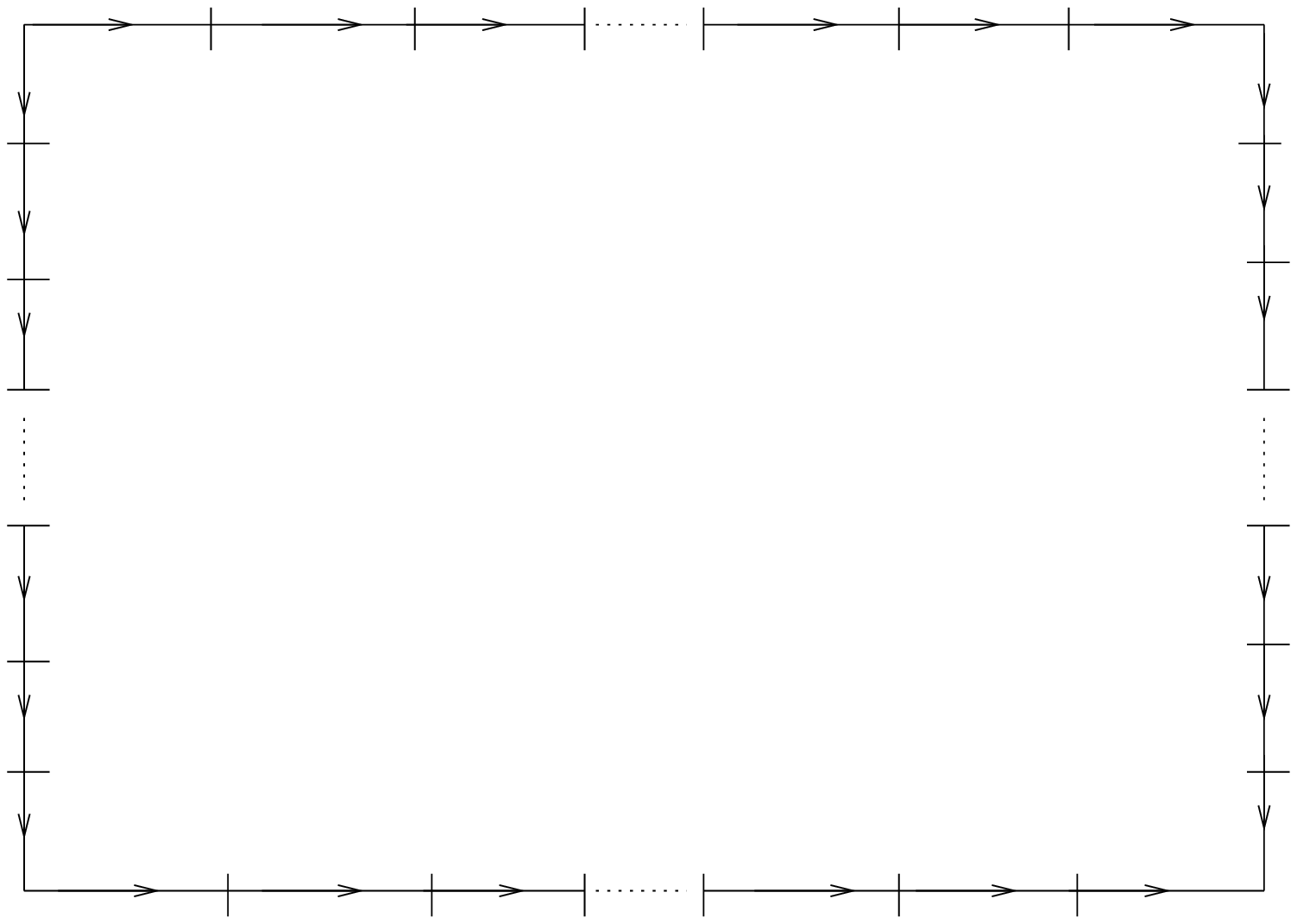}%
\end{picture}%
\setlength{\unitlength}{4144sp}%
\begingroup\makeatletter\ifx\SetFigFont\undefined%
\gdef\SetFigFont#1#2#3#4#5{%
  \reset@font\fontsize{#1}{#2pt}%
  \fontfamily{#3}\fontseries{#4}\fontshape{#5}%
  \selectfont}%
\fi\endgroup%
\begin{picture}(7137,5145)(1516,-6385)
\put(1666,-3121){\makebox(0,0)[lb]{\smash{{\SetFigFont{12}{14.4}{\rmdefault}{\mddefault}{\updefault}{\color[rgb]{0,0,0}$g_2$}%
}}}}
\put(1666,-2446){\makebox(0,0)[lb]{\smash{{\SetFigFont{12}{14.4}{\rmdefault}{\mddefault}{\updefault}{\color[rgb]{0,0,0}$f^{\delta_1}$}%
}}}}
\put(1666,-1906){\makebox(0,0)[lb]{\smash{{\SetFigFont{12}{14.4}{\rmdefault}{\mddefault}{\updefault}{\color[rgb]{0,0,0}$g_1$}%
}}}}
\put(1666,-4651){\makebox(0,0)[lb]{\smash{{\SetFigFont{12}{14.4}{\rmdefault}{\mddefault}{\updefault}{\color[rgb]{0,0,0}$g_k$}%
}}}}
\put(8596,-5731){\makebox(0,0)[lb]{\smash{{\SetFigFont{12}{14.4}{\rmdefault}{\mddefault}{\updefault}{\color[rgb]{0,0,0}$g_{k+1}$}%
}}}}
\put(4546,-6316){\makebox(0,0)[lb]{\smash{{\SetFigFont{12}{14.4}{\rmdefault}{\mddefault}{\updefault}{\color[rgb]{0,0,0}$b_2$}%
}}}}
\put(5941,-1411){\makebox(0,0)[lb]{\smash{{\SetFigFont{12}{14.4}{\rmdefault}{\mddefault}{\updefault}{\color[rgb]{0,0,0}$a_m$}%
}}}}
\put(8596,-1951){\makebox(0,0)[lb]{\smash{{\SetFigFont{12}{14.4}{\rmdefault}{\mddefault}{\updefault}{\color[rgb]{0,0,0}$g_1$}%
}}}}
\put(8596,-3076){\makebox(0,0)[lb]{\smash{{\SetFigFont{12}{14.4}{\rmdefault}{\mddefault}{\updefault}{\color[rgb]{0,0,0}$g_2$}%
}}}}
\put(8596,-2536){\makebox(0,0)[lb]{\smash{{\SetFigFont{12}{14.4}{\rmdefault}{\mddefault}{\updefault}{\color[rgb]{0,0,0}$f^{\delta_1}$}%
}}}}
\put(8596,-4516){\makebox(0,0)[lb]{\smash{{\SetFigFont{12}{14.4}{\rmdefault}{\mddefault}{\updefault}{\color[rgb]{0,0,0}$g_k$}%
}}}}
\put(8596,-5146){\makebox(0,0)[lb]{\smash{{\SetFigFont{12}{14.4}{\rmdefault}{\mddefault}{\updefault}{\color[rgb]{0,0,0}$f^{\delta_k}$}%
}}}}
\put(6841,-6316){\makebox(0,0)[lb]{\smash{{\SetFigFont{12}{14.4}{\rmdefault}{\mddefault}{\updefault}{\color[rgb]{0,0,0}$f^{\beta_n}$}%
}}}}
\put(5941,-6316){\makebox(0,0)[lb]{\smash{{\SetFigFont{12}{14.4}{\rmdefault}{\mddefault}{\updefault}{\color[rgb]{0,0,0}$b_n$}%
}}}}
\put(2431,-6316){\makebox(0,0)[lb]{\smash{{\SetFigFont{12}{14.4}{\rmdefault}{\mddefault}{\updefault}{\color[rgb]{0,0,0}$b_1$}%
}}}}
\put(3511,-6316){\makebox(0,0)[lb]{\smash{{\SetFigFont{12}{14.4}{\rmdefault}{\mddefault}{\updefault}{\color[rgb]{0,0,0}$f^{\beta_1}$}%
}}}}
\put(2296,-1411){\makebox(0,0)[lb]{\smash{{\SetFigFont{12}{14.4}{\rmdefault}{\mddefault}{\updefault}{\color[rgb]{0,0,0}$a_1$}%
}}}}
\put(3241,-1411){\makebox(0,0)[lb]{\smash{{\SetFigFont{12}{14.4}{\rmdefault}{\mddefault}{\updefault}{\color[rgb]{0,0,0}$f^{\alpha_1}$}%
}}}}
\put(4321,-1411){\makebox(0,0)[lb]{\smash{{\SetFigFont{12}{14.4}{\rmdefault}{\mddefault}{\updefault}{\color[rgb]{0,0,0}$a_2$}%
}}}}
\put(8011,-6316){\makebox(0,0)[lb]{\smash{{\SetFigFont{12}{14.4}{\rmdefault}{\mddefault}{\updefault}{\color[rgb]{0,0,0}$b_{n+1}$}%
}}}}
\put(1576,-5191){\makebox(0,0)[lb]{\smash{{\SetFigFont{12}{14.4}{\rmdefault}{\mddefault}{\updefault}{\color[rgb]{0,0,0}$f^{\delta_k}$}%
}}}}
\put(1531,-5731){\makebox(0,0)[lb]{\smash{{\SetFigFont{12}{14.4}{\rmdefault}{\mddefault}{\updefault}{\color[rgb]{0,0,0}$g_{k+1}$}%
}}}}
\put(6976,-1411){\makebox(0,0)[lb]{\smash{{\SetFigFont{12}{14.4}{\rmdefault}{\mddefault}{\updefault}{\color[rgb]{0,0,0}$f^{\alpha_m}$}%
}}}}
\put(7876,-1411){\makebox(0,0)[lb]{\smash{{\SetFigFont{12}{14.4}{\rmdefault}{\mddefault}{\updefault}{\color[rgb]{0,0,0}$a_{m+1}$}%
}}}}
\end{picture}%

%% file: bild5a.pspdftex
\begin{picture}(0,0)%
\includegraphics{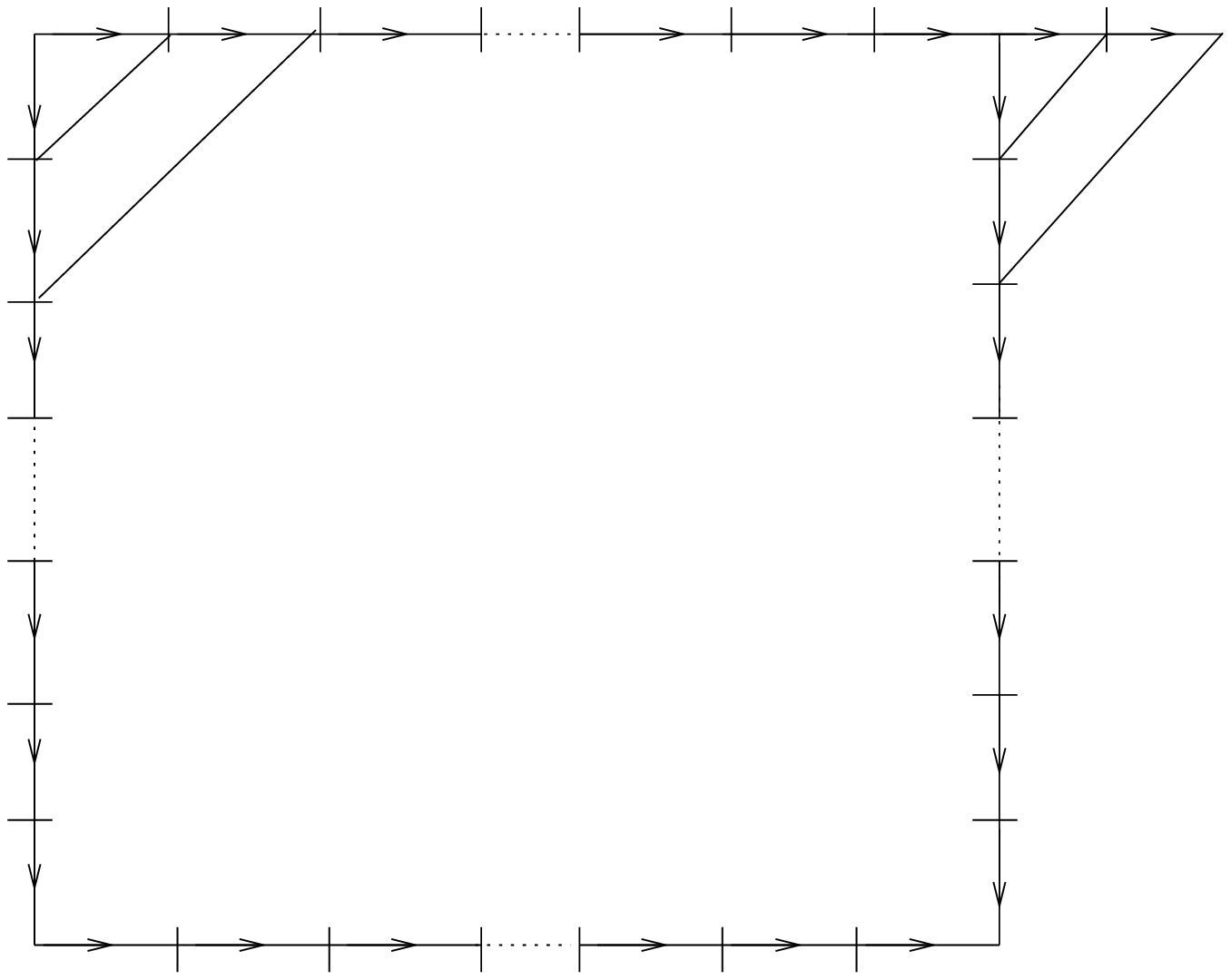}%
\end{picture}%
\setlength{\unitlength}{4144sp}%
\begingroup\makeatletter\ifx\SetFigFont\undefined%
\gdef\SetFigFont#1#2#3#4#5{%
  \reset@font\fontsize{#1}{#2pt}%
  \fontfamily{#3}\fontseries{#4}\fontshape{#5}%
  \selectfont}%
\fi\endgroup%
\begin{picture}(6417,5145)(1516,-6385)
\put(1666,-3121){\makebox(0,0)[lb]{\smash{{\SetFigFont{12}{14.4}{\rmdefault}{\mddefault}{\updefault}{\color[rgb]{0,0,0}$g_2$}%
}}}}
\put(1531,-5731){\makebox(0,0)[lb]{\smash{{\SetFigFont{12}{14.4}{\rmdefault}{\mddefault}{\updefault}{\color[rgb]{0,0,0}$g_{k+1}$}%
}}}}
\put(2206,-1411){\makebox(0,0)[lb]{\smash{{\SetFigFont{12}{14.4}{\rmdefault}{\mddefault}{\updefault}{\color[rgb]{0,0,0}$a_1$}%
}}}}
\put(3691,-1411){\makebox(0,0)[lb]{\smash{{\SetFigFont{12}{14.4}{\rmdefault}{\mddefault}{\updefault}{\color[rgb]{0,0,0}$a_2$}%
}}}}
\put(2926,-1411){\makebox(0,0)[lb]{\smash{{\SetFigFont{12}{14.4}{\rmdefault}{\mddefault}{\updefault}{\color[rgb]{0,0,0}$f^{\alpha_1}$}%
}}}}
\put(2161,-1771){\makebox(0,0)[lb]{\smash{{\SetFigFont{12}{14.4}{\rmdefault}{\mddefault}{\updefault}{\color[rgb]{0,0,0}$p$}%
}}}}
\put(6931,-1771){\makebox(0,0)[lb]{\smash{{\SetFigFont{12}{14.4}{\rmdefault}{\mddefault}{\updefault}{\color[rgb]{0,0,0}$p$}%
}}}}
\put(7561,-1411){\makebox(0,0)[lb]{\smash{{\SetFigFont{12}{14.4}{\rmdefault}{\mddefault}{\updefault}{\color[rgb]{0,0,0}$f^{\alpha_1}$}%
}}}}
\put(6346,-6316){\makebox(0,0)[lb]{\smash{{\SetFigFont{12}{14.4}{\rmdefault}{\mddefault}{\updefault}{\color[rgb]{0,0,0}$b_{n+1}$}%
}}}}
\put(5626,-6316){\makebox(0,0)[lb]{\smash{{\SetFigFont{12}{14.4}{\rmdefault}{\mddefault}{\updefault}{\color[rgb]{0,0,0}$f^{\beta_n}$}%
}}}}
\put(6481,-5191){\makebox(0,0)[lb]{\smash{{\SetFigFont{12}{14.4}{\rmdefault}{\mddefault}{\updefault}{\color[rgb]{0,0,0}$f^{\delta_k}$}%
}}}}
\put(6526,-3121){\makebox(0,0)[lb]{\smash{{\SetFigFont{12}{14.4}{\rmdefault}{\mddefault}{\updefault}{\color[rgb]{0,0,0}$g_2$}%
}}}}
\put(6481,-2491){\makebox(0,0)[lb]{\smash{{\SetFigFont{12}{14.4}{\rmdefault}{\mddefault}{\updefault}{\color[rgb]{0,0,0}$f^{\delta_1}$}%
}}}}
\put(6526,-1861){\makebox(0,0)[lb]{\smash{{\SetFigFont{12}{14.4}{\rmdefault}{\mddefault}{\updefault}{\color[rgb]{0,0,0}$g_1$}%
}}}}
\put(1576,-5236){\makebox(0,0)[lb]{\smash{{\SetFigFont{12}{14.4}{\rmdefault}{\mddefault}{\updefault}{\color[rgb]{0,0,0}$f^{\delta_k}$}%
}}}}
\put(6301,-1411){\makebox(0,0)[lb]{\smash{{\SetFigFont{12}{14.4}{\rmdefault}{\mddefault}{\updefault}{\color[rgb]{0,0,0}$a_{m+1}$}%
}}}}
\put(5581,-1411){\makebox(0,0)[lb]{\smash{{\SetFigFont{12}{14.4}{\rmdefault}{\mddefault}{\updefault}{\color[rgb]{0,0,0}$f^{\alpha_m}$}%
}}}}
\put(4906,-1411){\makebox(0,0)[lb]{\smash{{\SetFigFont{12}{14.4}{\rmdefault}{\mddefault}{\updefault}{\color[rgb]{0,0,0}$a_m$}%
}}}}
\put(3691,-6316){\makebox(0,0)[lb]{\smash{{\SetFigFont{12}{14.4}{\rmdefault}{\mddefault}{\updefault}{\color[rgb]{0,0,0}$b_2$}%
}}}}
\put(2926,-6316){\makebox(0,0)[lb]{\smash{{\SetFigFont{12}{14.4}{\rmdefault}{\mddefault}{\updefault}{\color[rgb]{0,0,0}$f^{\beta_1}$}%
}}}}
\put(2746,-2266){\makebox(0,0)[lb]{\smash{{\SetFigFont{12}{14.4}{\rmdefault}{\mddefault}{\updefault}{\color[rgb]{0,0,0}$q$}%
}}}}
\put(7156,-2221){\makebox(0,0)[lb]{\smash{{\SetFigFont{12}{14.4}{\rmdefault}{\mddefault}{\updefault}{\color[rgb]{0,0,0}$q$}%
}}}}
\put(6391,-5821){\makebox(0,0)[lb]{\smash{{\SetFigFont{12}{14.4}{\rmdefault}{\mddefault}{\updefault}{\color[rgb]{0,0,0}$g_{k+1}$}%
}}}}
\put(6976,-1411){\makebox(0,0)[lb]{\smash{{\SetFigFont{12}{14.4}{\rmdefault}{\mddefault}{\updefault}{\color[rgb]{0,0,0}$a_1$}%
}}}}
\put(4951,-6316){\makebox(0,0)[lb]{\smash{{\SetFigFont{12}{14.4}{\rmdefault}{\mddefault}{\updefault}{\color[rgb]{0,0,0}$b_n$}%
}}}}
\put(2206,-6316){\makebox(0,0)[lb]{\smash{{\SetFigFont{12}{14.4}{\rmdefault}{\mddefault}{\updefault}{\color[rgb]{0,0,0}$b_1$}%
}}}}
\put(1621,-2536){\makebox(0,0)[lb]{\smash{{\SetFigFont{12}{14.4}{\rmdefault}{\mddefault}{\updefault}{\color[rgb]{0,0,0}$f^{\delta_1}$}%
}}}}
\put(6526,-4561){\makebox(0,0)[lb]{\smash{{\SetFigFont{12}{14.4}{\rmdefault}{\mddefault}{\updefault}{\color[rgb]{0,0,0}$g_k$}%
}}}}
\put(1666,-1861){\makebox(0,0)[lb]{\smash{{\SetFigFont{12}{14.4}{\rmdefault}{\mddefault}{\updefault}{\color[rgb]{0,0,0}$g_1$}%
}}}}
\put(1666,-4561){\makebox(0,0)[lb]{\smash{{\SetFigFont{12}{14.4}{\rmdefault}{\mddefault}{\updefault}{\color[rgb]{0,0,0}$g_k$}%
}}}}
\end{picture}%